\documentclass{amsart}

\usepackage{amsmath}
\usepackage{amsthm}
\usepackage{hyperref}
\usepackage{amsfonts,graphics,amsthm,amsfonts,amscd,latexsym}
\usepackage{epsfig}
\usepackage{flafter}
\usepackage{mathtools}
\usepackage{comment}
\usepackage{stmaryrd}

\usepackage{mathabx,epsfig}

\hypersetup{
    colorlinks=true,    
    linkcolor=blue,          
    citecolor=blue,      
    filecolor=blue,      
    urlcolor=blue           
}
\usepackage{tikz}
\usetikzlibrary{graphs,positioning,arrows,shapes.misc,decorations.pathmorphing}

\tikzset{
    >=stealth,
    every picture/.style={thick},
    graphs/every graph/.style={empty nodes},
}

\tikzstyle{vertex}=[
    draw,
    circle,
    fill=black,
    inner sep=1pt,
    minimum width=5pt,
]
\usepackage[position=top]{subfig}
\usepackage{amssymb}
\usepackage{color}

\calclayout

\usetikzlibrary{decorations.pathmorphing}
\tikzstyle{printersafe}=[decoration={snake,amplitude=0pt}]

\newcommand{\supp}{\operatorname{supp}}

\newcommand{\pp}{\mathbb{P}}

\newcommand{\qq}{\mathbb{Q}}
\newcommand{\zz}{\mathbb{Z}}

\newcommand{\rr}{\mathbb{R}}
\newcommand{\cc}{\mathbb{C}}
\newcommand{\kk}{\mathbb{K}}

\def\O#1.{\mathcal {O}_{#1}}			
\def\pr #1.{\mathbb P^{#1}}				
\def\af #1.{\mathbb A^{#1}}			
\def\ses#1.#2.#3.{0\to #1\to #2\to #3 \to 0}	
\def\xrar#1.{\xrightarrow{#1}}			
\def\K#1.{K_{#1}}						
\def\bA#1.{\mathbf{A}_{#1}}			
\def\bM#1.{\mathbf{M}_{#1}}				
\def\bL#1.{\mathbf{L}_{#1}}				
\def\bB#1.{\mathbf{B}_{#1}}				
\def\bK#1.{\mathbf{K}_{#1}}			
\def\subs#1.{_{#1}}					
\def\sups#1.{^{#1}}

\usepackage{tikz}
\usetikzlibrary{matrix,arrows,decorations.pathmorphing}

  \newtheorem{theorem}{Theorem}[section]
  \newtheorem{lemma}[theorem]{Lemma}
  \newtheorem{proposition}[theorem]{Proposition}
  \newtheorem{corollary}[theorem]{Corollary}
  
  \newtheorem{conjecture}[theorem]{Conjecture}

  \newtheorem{definition}[theorem]{Definition}
  \newtheorem{example}[theorem]{Example}
  
  \newtheorem{problem}[theorem]{Problem}
  \newtheorem{question}[theorem]{Question}

\newtheorem{remark}[theorem]{Remark}

\theoremstyle{remark}

\numberwithin{equation}{section}

\usepackage[all]{xy}

\begin{document}

\title[Coregularity of Fano varieties]{Coregularity of Fano varieties}

\author[J.~Moraga]{Joaqu\'in Moraga}
\address{Department of Mathematics, Princeton University, Fine Hall, Washington Road, Princeton, NJ 08544-1000, USA
}
\email{jmoraga@princeton.edu}

\maketitle

\begin{abstract}
The regularity of a Fano variety, denoted by ${\rm reg}(X)$,
is the largest dimension of the dual complex 
of a log Calabi--Yau structure on $X$.
The coregularity is defined to be 
\[
{\rm coreg}(X):=
\dim X - {\rm reg}(X)-1.
\]
The coregularity 
is the complementary dimension of the regularity.
We expect that the coregularity of a Fano variety governs, to a large extent, the geometry of $X$.
In this note, we review the history of Fano varieties, give some examples,
survey some important theorems, introduce the coregularity,
and propose several problems regarding 
this invariant of Fano varieties.
\end{abstract}

\setcounter{tocdepth}{1} 
\tableofcontents

\section{Preliminaries}

In this note, we introduce and study the coregularity of Fano varieties.
In this section, we present some basic objects that will be used to introduce the invariant:
canonical line bundle, Fano manifolds, Fano varieties, Calabi--Yau pairs, 
theory of complements, and dual complexes.

\subsection{Canonical line bundle} 
One of the main aims of algebraic geometry is to classify 
smooth projective varieties. 
Given a $n$-dimensional 
smooth projective variety $X\subset \pp^N$, 
we can define its {\em tangent bundle} $T_X$. 
Dualizing the tangent bundle, we obtain 
the {\em cotangent bundle} $\Omega_X:=T_X^*$.
In some sense, the tangent bundle and the cotangent bundle, 
are the only natural vector bundles that we can associate
with any smooth projective variety. 
The {\em canonical line bundle}, denoted by $\omega_X$, 
is the $n$-th exterior power of the cotangent bundle, i.e., 
$\omega_X := \bigwedge^n \Omega_X$.
Over the complex numbers, 
$\omega_X$ is the determinant bundle 
of holomorphic forms on $X$.
The previous bundles are independent of the chosen projective embedding of $X$.
The {\em canonical divisor} $K_X$ is a divisor on $X$ for which $\omega_X\simeq \mathcal{O}_X(K_X)$.
There are three pure classes of smooth projective varieties, 
depending on the positivity or negativity of its
canonical line bundle:
\begin{enumerate} 
    \item A smooth projective variety $X$ is said to be {\em Fano} if 
    $\omega_X$ is anti-ample, i.e., $\omega_X\cdot C<0$ for every curve $C\subset X$.
    \item A smooth projective variety $X$ is said to be {\em Calabi--Yau} if 
    $\omega_X$ is numerically trivial, i.e., $\omega_X \cdot C=0$ for every curve $C\subset X$.
    \item A smooth projective variety $X$ is said to be {\em canonically polarized} if $\omega_X$ is ample, i.e., $\omega_X\cdot C>0$ for every curve $C\subset X$.
\end{enumerate}
The previous trichotomy generalizes to higher-dimensional varieties
the classic trichotomy of Riemann surfaces: 
the Riemann sphere, complex tori, and curves of genus at least two.
A smooth hypersurface $H\subset \pp^N$ of degree $d$ is
Fano (resp. Calabi--Yau and canonically polarized) 
if $d<n+1$ (resp. $d=n+1$ and $d>n+1$).
Among algebraic varieties, we expect that Fano varieties are easier to understand due to multiple reasons.
In dimension one there is only one them, so the initial expectation is that there are fewer Fano varieties than Calabi--Yau and canonically polarized varieties in any given dimension.
On the other hand, they tend to behave similarly to a projective space.
They tend to be simply connected, covered by holomorphic Riemann spheres,
and rigid under small deformations.
In this note, we focus on the study of Fano varieties, 
although this often intertwinds with a better understanding of both
Calabi--Yau and canonically polarized varieties.

\subsection{Fano manifolds}
The classification 
of smooth surfaces
with ample anti-canonical divisor
was achieved by Pasquale del Pezzo in 1887.
These surfaces are currently known as
{\em del Pezzo surfaces}
and they form $9$ families
depending on the degree of the anti-canonical divisor $-K_X$.
The concept of {\em Fano manifolds}
was studied by Leonard Roth~\cite{Roth1930}
and by Gino Fano~\cite{Fano1942} in the 30's and 40's, respectively.
The latter formally introduces the concept of Fano manifolds. 
Thus, del Pezzo surfaces
are smooth Fano surfaces.
During the 50's and the 60's there were some 
mentions to Fano manifolds, 
but it was nothing near a central topic in algebraic geometry.
The first systematic study of Fano varieties 
was led by Iskovskih in~\cite{Isk78a,Isk78b}.
In this paper, Iskovskih combines the recent techniques
introduced by Grothendieck to algebraic geometry 
and Fano's brilliant intuition to give a complete
treatment of smooth Fano $3$-folds.
In these papers, the main invariant to classify smooth Fano $3$-folds is the index. 
The index of a Fano manifold $X$ is defined to be:
\[
i(X):=\max\{ r\in \zz_{\geq 1} \mid 
rH \sim K_X \text{ for $H$ ample Cartier on $X$}
\}. 
\]
The main gadget to study these manifolds
was to find smooth elements $S\in |H|$.
The existence of these smooth elements was proved by Shokurov~\cite{Sho79}.

The index of a Fano manifold of dimension $3$ is an integer in $\{1,2,3,4\}$.
If $i(X)=4$, then $X\simeq \pp^3$.
If $i(X)=3$, then $X$ is isomorphic to a quadric in $\pp^4$.
If $i(X)=2$, then 
$H^3\in \{1,\dots,7\}$. 
If $H^3\in \{5,6,7\}$, then there is a unique Fano $3$-fold $V_d\subset \pp^{d+1}$.
If $H^3\in \{1,2,3,4\}$ gives weighted hypersurfaces or complete intersections.
Most Fano manifolds have index $i(X)=1$
and the anti-canonical map, defined by $|-K_X|$,
is a morphism $\phi_{-K_X}\colon X\rightarrow \pp^N$.
In what follows, we set $2g-2=-K_X^3$.
There are three possible behaviours in the case of index one:
\begin{itemize}
    \item the anti-canonical map $\phi_{-K_X}$ is an embedding into
    $\pp^{2g-2}$ and the image is an intersection of quadrics;
    \item {\em hyperelliptic threefolds}: the anti-canonical map $\phi_{-K_X}$ maps $X$ into $\pp^{g+1}$
    with a degree two map onto its image; and 
    \item {\em trigonal threefolds}: the image of the anti-canonical map $\phi_{-K_X}$ embeds 
    $X$ in $\pp^{g+1}$ and its image is contained in a $4$-fold scroll.
\end{itemize}
Using the previous description, 
Iskovskih gave a classification of smooth Fano $3$-folds of Picard rank one.
In~\cite{Bat81}, Batyrev classified smooth toric Fano $3$-folds.
In~\cite{MM82}, Mori and Mukai classified $3$-dimensional Fano manifolds
with second Betti number of at least two.
Conte started the study of Fano $3$-folds with Gorenstein singularities in~\cite{Cont85}.
Soon later, Maeda considered log Fano $3$-folds, i.e., log smooth pairs $(X,\Delta)$
for which $-(K_X+\Delta)$ was ample~\cite{Mae86}.
This was the first approximation to the modern definition of Fano varieties.
Afterwards, Fano varieties were studied from their topological perspective~\cite{Tsu88}, 
arithmetic perspective~\cite{FMT89}, the existence of K\"ahler-Einstein metrics~\cite{Mab87},
birational transformations~\cite{Tak89}, and characteristic p methods~\cite{Bal89}.

The existence of rational curves on Fano varieties is one of the cornerstones of the minimal model program~\cite{Mor82}.
In~\cite{KMM92}, Koll\'ar, Miyaoka, and Mori proved that two general points on a $n$-dimensional Fano manifold can be connected by a rational curve $C$ for which $-K_X\cdot C \leq n(n+1)$, i.e.,
Fano manifolds are rationally connected.
As a consequence, they prove that the inequality $-K_X^n\leq (n(n+1))^n$ holds and that 
there are only finitely many deformation families of $n$-dimensional Fano manifolds.
In~\cite{BB92}, the Borisov brothers studied toric Fano varieties, i.e., 
a toric variety $X$ for which $-K_X$ is ample (with no smoothness condition).
They prove that in each dimension, there are only finitely many toric Fano varieties
for which a certain singularity invariant is controlled. 
This leads us to the introduction of, possibly singular, Fano varieties. 

\subsection{Fano varieties}
A Fano variety is the singular generalization of a Fano manifold. 
However, in order to still get a nice theory, one needs to define the class
of allowed singularities carefully.

\begin{example}\label{ex:1}
{\em 
Let $G$ be a finite group
acting on a Fano manifold $X$, let $Y:=X/G$ be the quotient, and 
$\pi\colon X\rightarrow Y$ be the quotient morphism.
Then, by the Riemann-Hurwitz formula, we can write 
\[
K_X=\pi^*(K_Y+\Delta), 
\]
where $\Delta$ is an effective divisor on $X$ with standard coefficients, i.e., 
coefficients of the form $1-\frac{1}{n}$ for some positive integer $n$.
Note that $K_X+\Delta$ is a $\qq$-divisor, but it is still $\qq$-Cartier, i.e., 
$m(K_X+\Delta)$ is a honest Cartier divisor for a sufficiently divisible integer $m$.
As a $\qq$-divisor, we have that $-(K_X+\Delta)$ is ample, i.e., $-m(K_X+\Delta)$ is a very ample Cartier divisor for some $m$ divisible enough.
}
\end{example}

The previous example suggests that we should generalize the definition of Fano manifolds in two directions: 
we should allow boundary divisors
and possibly singular varieties. 

A {\em pair} consists of the data 
$(X,\Delta)$ of a normal quasi-projective variety
$X$
and an effective divisor $\Delta$ on $X$ for which $K_X+\Delta$ is $\qq$-Cartier. 
To define classes of singularities for our pair, we introduce the concept of log discrepancies. 
Let $\phi\colon Y\rightarrow X$ be a projective birational morphism and $E\subset Y$ be a prime divisor.
The {\em log discrepancy} of $(X,\Delta)$ at $E$, is defined to be
\[
a_E(X,\Delta)=1-{\rm coeff}_E(K_Y-\phi^*(K_X+\Delta)).
\]
In a few words, $a_E(X,\Delta)$ measures how singular is $(X,\Delta)$ along the tangent directions corresponding to the points of $E$.
Thus, nice singularities should correspond to larger values of $a_E(X,\Delta)$ for every possible $E$. 
In the previous context, we say that $E$ is a prime divisor over $X$.
If $E$ lies in the exceptional locus of $\phi$ (where the morphism is not an isomorphism), then we
say that $E$ is exceptional over $X$.
This leads to the following definitions: 
\begin{enumerate}
    \item We say that $(X,\Delta)$ is {\em terminal} if $a_E(X,\Delta)>1$ for every prime divisor $E$ exceptional over $X$.
    \item We say that $(X,\Delta)$ is {\em canonical}
    if $a_E(X,\Delta)\geq 1$ for every prime divisor $E$ over $X$.
    \item We say that $(X,\Delta)$ is {\em log terminal} (or {\em Kawamata log terminal}) if $a_E(X,\Delta)>0$ for every prime divisor $E$ over $X$. 
    We write klt for short.
    \item We say that $(X,\Delta)$ is {\em log canonical} if $a_E(X,\Delta)\geq 0$ for every prime divisor $E$ over $X$.
    We write lc for short.
\end{enumerate}
The class of terminal singularities is the largest class of singularities obtained by the minimal model program for smooth varieties.
On the other hand, log canonical singularities
are the wider class of singularities in which we expect that the MMP can be run.
Quotient singularities are klt~\cite{Kol13}. 
Symplectic singularities are canonical~\cite{Rei87}.
The Reid-Tai criterion allows us to discern whether a quotient singularity is canonical or terminal~\cite{Rei85}.
The affine cone over a rational normal curve is also klt.
On the other hand, a cone over an elliptic curve is lc but not klt (see, e.g.,~\cite{Kol13}).
The boundary $\Delta$ can often be thought of as a correction term for the canonical divisor $K_X$ not being $\qq$-Cartier.
We say that a variety $X$ is {\em klt type} 
(resp {\em lc type})
if there exists a boundary $\Delta$ for which
$(X,\Delta)$ is klt (resp. lc).

We say that variety $X$ is {\em Fano} (also known as $\qq$-Fano) if $X$ has klt singularities
and $-K_X$ is ample.
We say that a variety $X$ is {\em of Fano type}
if there exists a boundary $\Delta$ on $X$ for which $(X,\Delta)$ has klt singularities
and $-(K_X+\Delta)$ is ample. 
From Example~\ref{ex:1}, we deduce that every finite quotient of a Fano manifold
is of Fano type.
Furthermore, every projective toric variety is of Fano type.
Even if del Pezzo surfaces are classified, 
we can find far more Fano varieties of dimension two (see, e.g.,~\cite{OZ99,MZ04}).

The books by Prokhorov and Iskovskikh~\cite{IP99a,IP99b} 
gives a very good account for the study of Fano varieties, especially those of dimension three.
Many results regarding rationally connectedness, rigidity,
and simply connectedness of Fano manifolds
were generalized to the Fano type setting~\cite{Zha95,Tak00,Zha06,dFH11}.
Fano varieties have also been extensively studied from the perspective of birational rigidity~\cite{Puk02,Che06}.
It is worth mentioning that Fano varieties are tightly 
related to mirror symmetry~\cite{CL10}.
In~\cite{Ale94}, Alexeev proved that Fano type surfaces with mild singularities form bounded families. 

\subsection{Calabi--Yau pairs} 
Analogously, a Calabi--Yau pair is the singular generalization of a Calabi--Yau manifold. 
Again, we consider pairs instead of varieties.

A {\em log Calabi--Yau pair}
is a pair $(X,\Delta)$ with log canonical singularities
for which $K_X+\Delta\equiv 0$, i.e.,
the intersection of every curve with the
$\qq$-divisor $K_X+\Delta$ is zero.
We say that a variety $X$ is of {\em log Calabi--Yau type}
if there exists a boundary $\Delta$ on $X$ for which
the pair $(X,\Delta)$ is log Calabi--Yau. 
Every Calabi--Yau manifold is already a 
variety of log Calabi--Yau type.
Analogously as in Example~\ref{ex:1}, the 
finite quotient of a Calabi--Yau manifold
is a variety of log Calabi--Yau type.

Every Fano type variety is log Calabi--Yau.
Indeed, we can proceed as follows: 
if $-(K_X+\Delta)$ is an ample divisor, then
we can take a large multiple for which 
$-m(K_X+\Delta)$ is a very ample divisor.
So we can consider a general section 
\begin{equation}\label{eq:comp} 
\Gamma \in |-m(K_X+\Delta)|
\end{equation} 
which intersects the support of $\Delta$ transversally. 
By choosing $\Gamma$ general enough, we achieve that
the pair $(X,\Delta+\Gamma/m)$ is klt.
Furthermore, by construction, we have that 
\[
K_X+\Delta+\Gamma/m\equiv 0.
\]
The previous argument is in general not very effective, 
as we have no control over the integer $m$ which we need for this construction to work.

\begin{example} 
{\em 
For instance, if we consider the projective plane 
$\pp^2$, there are several ways to produce a boundary $\Gamma$ for which $(\pp^2,\Gamma)$ is log Calabi--Yau.
We can take $\Gamma$ to be:
\begin{enumerate}
    \item the sum of three transversal lines, 
    \item the sum of a line and a conic intersecting transversally, or 
    \item an elliptic curve.
\end{enumerate}
Of course, if we allow fractional coefficients on the boundary (i.e, consider elements of $|-mK_{\pp^2}|$ for $m\geq 2$), then we can construct far more examples.
}
\end{example} 

In general, we try to seek for those boundaries $\Gamma$ which have more components and smaller denominators in their coefficients.
In other words, that belong to $|-m(K_X+\Delta)|$
for smaller values of $m$.

The singularities introduced in the previous subsection:
Kawamata log terminal and log canonical
are the local versions of 
Fano type
and log Calabi--Yau type varieties. 
Indeed, the affine cone over a Fano type variety is of klt type
and the affine cone over a log Calabi--Yau type variety
is of lc type.

\begin{remark}
{\em 
When we work with singularities of a pair, i.e., 
a precise closed point $x\in (X,\Delta)$, we may write 
$(X,\Delta;x)$ to mean that the statement that we are writing holds for the pair $(X,\Delta)$ around the closed point $x$.
}
\end{remark}

In the other direction, 
given a klt singularity $(X,\Delta;x)$ (resp. lc singularity), we can find a projective birational morphism that contracts a unique prime divisor $E$ mapping onto $x$ so that $E$ admits the structure of a Fano type (resp. log Calabi--Yau type) variety. 
The previous statements are proved in~\cite{Pro00,Kud01,Xu14}.
The understanding of klt singularities has been closely intertwined with the understanding of Fano varieties.

\subsection{Theory of complements} 
As pointed out in the previous subsection, 
every Fano type variety 
admits a boundary which turns it into 
a log Calabi--Yau pair.
This choice always exists but it is highly non-unique. One of the aim goals of the theory of complements
is to make this choice effective 
in terms of certain invariants of $X$. 
The second aim is to make this choice in a somewhat canonical way, or at the very least, find certain
distinguished log Calabi--Yau structures the Fano type variety.

A boundary $B$ with $\qq$-coefficients 
on a normal projective variety $X$ 
for which $(X,B)$ is log Calabi--Yau, 
is called a {$\qq$-complement}.
A $\qq$-complement is said to be a {\em $N$-complement}
if the linear equivalence
\[
N(K_X+B)\sim 0
\]
holds. 
This means that $NB\in |-NK_X|$.
Note that in the previous definition we are not requiring the divisor $K_X$ to be Cartier, 
not even that $NK_X$ is Cartier. 
Hence, $B\sim -NK_X$ is 
linear equivalence of Weil divisors, i.e., 
their difference is the divisor associated
to a rational function on $X$. 
Note that this notion even works for $\qq$-divisors.

The theory of complements was initiated by
Prokhorov and Shokurov in the early 2000's.
In~\cite{Sho00}, Shokurov proved that surfaces
of Fano type almost always admit 
a $N$-complement 
where $N\in \{1,2,3,6\}$.
More precisely, Shokurov proves the following theorem.

\begin{theorem}\label{thm:boundedness-complements-surfaces}
Let $X$ be a Fano type surface. 
Then, one of the following statements hold:
\begin{enumerate}
    \item the variety $X$ admits a $N$-complement where
    $N\in \{1,2,3,6\}$, or
    \item for every $B$ such that $(X,B)$ is log Calabi-Yau, the pair $(X,B)$ has klt singularities.
\end{enumerate}
\end{theorem}

In other words, either we can find a 
$6$-complement
or it is not possible to produce a strictly
log canonical singularity on $(X,B)$.
By Alexeev's boundedness of Fano surfaces~\cite{Ale94},
the Fano surfaces in the second class of Theorem~\ref{thm:boundedness-complements-surfaces}
belong to a bounded family, i.e., 
they can be described using 
a fixed number of variables and polynomial equations. 
These surfaces are known as
{\em exceptional Fano type surfaces}. 
Although exceptional Fano type surfaces belong to a bounded family, 
until now we do not have a complete classification of these surfaces.
Nevertheless, the fact that they belong to a bounded family 
implies that there exists a
positive integer $N$ for which all
exceptional Fano type surfaces
admit a $N$-complement. 
This implies the following theorem, due to Shokurov.

\begin{theorem}
There exists a positive integer $N_2$ satisfying the following.
Let $X$ be a Fano type surface. 
Then, $X$ admits a $N_2$-complement.
\end{theorem}

In~\cite{Pro01}, there is a good account of Shokurov's proof of the boundedness of complements for surfaces.
and some further results are proved.
As of today, it is unknown what is the optimal value for $N_2$.
However, it is expected that $N_2=66$ suffices
and this number is related to Sylvester's sequence as we explain below.
The previous theorem motivated the problem known as ``boundedness of complements", i.e., 
find a constant $N_n$ only depending on the dimension $n$, 
such that every Fano type variety $X$ of dimension
$n$ admits a $N_n$-complement.
This problem is also known as the ``existence of bounded complements". 
In some words, the problem asks about
an effective way of constructing 
a log Calabi--Yau variety in a Fano type variety.
In~\cite{PS01,PS09}, Prokhorov and Shokurov started an approach
to tackle the existence of bounded complements. 
The strategy consists of two main steps:
\begin{itemize}
    \item lifting complements from log canonical places, and
    \item lifting complements from the base of
    the fibration by using the canonical bundle formula.
\end{itemize}
Hence, by running a suitable argument, one could aim to lift the so-called bounded complement from lower-dimensional Fano type varieties.
In~\cite{Bir19}, Birkar settled the 
existence of bounded complements for $n$-dimensional Fano type varieties. 

\begin{theorem}\label{thm:boundedness-complements}
There exists a constant $N_n$, only depending on $n$, satisfying the following.
Let $X$ be a Fano type variety of dimension $n$.
Then, there exists a $N_n$-complement.
\end{theorem}

This achievement used crucially 
some deep theorems in
the minimal model program~\cite{BCHM10,HMX14}.
Moreover, it is a vital step for the boundedness
of Fano varieties~\cite{Bir21}.
A similar statement holds for klt singularities 
via global-to-local arguments (see subsection~\ref{subsec:fano-vs-klt}).

\begin{theorem}
There exists a constant $N_n$, 
only depending on $n$, satisfying the following.
Let $(X;x)$ be a $n$-dimensional klt singularity.
Then, there exists a boundary $B$ on $X$ for which
$(X,B;x)$ is log canonical, 
$x$ is a log canonical center of $(X,B)$, 
and $N(K_X+B)\sim 0$ holds on a neighborhood of the point $x\in X$.
\end{theorem}

The boundary $B$ constructed in the previous statement is called a $N$-complement
of the klt singularity $(X;x)$.
It is worth mentioning, that we expect the 
existence of bounded complements 
to hold for log Calabi-Yau varieties as well. 
More precisely, 
whenever a variety $X$ admits a log Calabi--Yau structure it also admits an effective log Calabi--Yau structure.
In the language of complements: 
if a $n$-dimensional variety $X$ admits a $\qq$-complement, 
then it admits a $N_n$-complement.
In~\cite{FMX19}, the authors proved this statement in dimension at most three. 

\begin{theorem}
There exists a constant $N_3$ satisfying the following.
Let $X$ be a normal projective variety
of log Calabi-Yau type
and dimension at most $3$.
Then, $X$ admits a $N_3$-complement.
\end{theorem}

The existence of bounded complements
for log Calabi-Yau varieties
of dimension at least $4$ remains open.
The theory of complements helped to make big breakthroughs in the study of Fano varieties: 
The boundedness of Fano varieties with mild singularities~\cite{Bir19,Bir21b} 
and the algebro-geometric K-stability theory~\cite{LXZ21}. 
In the latter, the authors introduced
the concept of special complements
which behaves well with respect to the normalized volume. 
It is expected that further understanding of complements
will be crucial for the development of Fano type varieties
and klt singularities.
In what follows, we propose some new special classes of complements related to dual complexes of Calabi-Yau type varieties.

\subsection{Dual complexes}
\label{subsec:DC}

A dual complex 
is a combinatorial object that encrypts the data of the intersection of divisors on a normal variety.

Given a normal variety $X$ 
and a reduced divisor
$E=\sum_{i\in I} E_i$
with simple normal crossing support.
We define the dual complex 
$\mathcal{D}(E)$ 
to be the CW complex whose 
$k$-dimensional cells $v_W$ correspond to irreducible components $W$
of $\bigcap_{j\in J}E_j$
where $J\subset I$ and
$|J|=k+1$.
Given $j_0\in J$,
the irreducible variety $W$
is contained in a unique irreducible component $Z$
of $\bigcap_{J\setminus \{j_0\} } E_j$. 
This inclusion induces a gluing
map $v_Z\hookrightarrow v_W$.
The {\em dimension} of the dual complex $\mathcal{D}(E)$
is the largest dimension of its cells.

Now, we turn to define the dual complex of a log Calabi--Yau pair $(X,B)$.
The naive idea is to define it as the dual complex of $\lfloor B\rfloor$. 
However, there are two issues with this. 
First, this definition does not behave well under birational modifications
as a blow-up of $X$ may introduce new divisors with coefficient one in the boundary of the log Calabi--Yau pair.
On the other hand, 
the divisor
$\lfloor B \rfloor$ may not have simple normal crossing support.
The concept of divisorially log terminal models was introduced, in part, to fix these two issues.

Let $(X,B)$ be a log canonical pair. 
A {\em log canonical place}
of $(X,B)$ is a prime divisor
$E$ over $X$ for which 
$a_E(X,B)=0$. 
We say that $(X,B)$ is {\em strictly log canonical}
if it is log canonical and admits a log canonical place. 
A {\em log canonical center}
of $(X,B)$ is the center on $X$
of a log canonical place of $(X,B)$.
A log canonical pair $(X,B)$
is said to be 
{\em divisorially log terminal}
(or {\em dlt}) if there exists an open subset $U\subseteq X$ satisfying the following conditions:
\begin{enumerate}
    \item the divisor
    $\lfloor B \rfloor|_U$ has simple normal crossing support, 
    \item every log canonical center of $(X,B)$ intersects $U$, and 
    \item the log canonical centers of $(X,B)$ are given
    by strata of $\lfloor B\rfloor$.
\end{enumerate}
In a few words, the pair
$(X,B)$ is dlt if it looks
simple normal crossing 
at the generic point of every log canonical center.
If $(X,B)$ is dlt,
then blowing up 
a stratum of $\lfloor B\rfloor$
induces a barycenter subdivision
on one of the simplices
of the dual complex $\mathcal{D}(\lfloor B\rfloor)$.
Hence, the dual complex 
of $(X,B)$, seen as a topological space, 
stabilizes
when the pair has dlt singularities. 
The following theorem due to Hacon allows us to birationally transform log canonical pairs
into dlt pairs (see, e.g.,~\cite{KK10}).

\begin{theorem}
Let $(X,B)$ be a log canonical pair.
There exists a projective birational morphism
$\phi\colon Y\rightarrow X$ 
satisfying the following conditions:
\begin{enumerate}
    \item the pair $(Y,B_Y)$ has dlt singularities,
    where $\phi^*(K_X+B)=K_Y+B_Y$, and 
    \item the projective birational morphism $\phi$ only contracts divisors with log discrepancy zero
    with respect to $(X,B)$.
\end{enumerate}
\end{theorem}

The previous theorem is known as the existence of dlt modifications. 
The dlt modification
allows us to define the dual complex of a log Calabi--Yau pair as follows.
This motivates the following definition. 

Let $(X,B)$ be a log Calabi--Yau pair. 
Let $\phi \colon Y\rightarrow X$ be a dlt modification of $(X,B)$.
We define 
the {\em dual complex}
$\mathcal{D}(X,B)$
of $(X,B)$ to be
$\mathcal{D}(\lfloor B_Y\rfloor)$.
A priori, the dual complex 
depends on the chosen dlt modification.
Indeed, as explained before, we can obtain different triangulations by 
performing more blow-ups.
However, De Fernex, Koll\'ar, and Xu proved that 
two such dual complexes
are simple-homotopy equivalent (see, e.g.,~\cite{dFKX17}).

\begin{theorem}
Let $(X,B)$ be a log Calabi-Yau pair. 
Then its dual complex
$\mathcal{D}(X,B)$ is well-defined up
to simple-homotopy equivalence.
\end{theorem}

Is expected that the dual complexes of log Calabi--Yau pairs are PL-homeomorphic to PL-spheres up to a finite cover.
More precisely, we have the following conjecture.

\begin{conjecture}\label{conj:dual-comp}
Let $(X,B)$ be a $n$-dimensional log Calabi-Yau pair.
There exists a finite cover $\pi\colon Y \rightarrow X$
such that
$\mathcal{D}(Y,B_Y)\simeq_{\rm PL} S^{n-1}$.
\end{conjecture}

By~\cite{KX16}, this conjecture
is known for varieties
of dimension at most $4$.
This conjecture is also known for Mori fiber spaces
whenever the base has dimension at most two
or when the total space has Picard at most two,
by the work of Mauri~\cite{Mau20}.
By~\cite{Mor21}, there are only finitely many possible fundamental groups
for the dual complexes
of $n$-dimensional log Calabi--Yau threefolds. 
Furthermore, by~\cite{Bra19}, these fundamental groups are finite.
Log Calabi--Yau varieties
also appear
when compactifying character varieties. 
In this direction,
the previous 
conjecture has been
studied
by Simpson~\cite{Sim16}
and Goldman and Toledo~\cite{GT10}.
Even though the previous conjecture is known up to dimension $4$, 
very little is known about the possible triangulations of the sphere
that we can obtain in dual complexes
of log Calabi--Yau pairs.
Any smooth lattice polytope can be achieved. 
Indeed, we can consider the projective toric variety corresponding to its dual fan with
its reduced toric boundary.

The following definition will be useful to state some of the results in this article.

\begin{definition}
{\em 
Let $(X,\Delta)$ and $(X',\Delta')$ be two log pairs
such that $X$ and $X'$ are birational.
We say that $(X,\Delta)$
and $(X',\Delta')$ are
{\em log crepant equivalent} 
if there is a common resolution 
$p\colon Z\rightarrow X$
and $q\colon Z\rightarrow X'$
for which the equality
\[
p^*(K_X+\Delta) = 
q^*(K_{X'}+\Delta')
\] 
holds. 
Observe that log crepant pairs share many common properties.
For instance, they have the same 
log discrepancies,
$(X,\Delta)$ is log Calabi--Yau if and only if $(X',\Delta')$ is log Calabi--Yau.
In this case, they have the same coregularity and dual complex.
}
\end{definition}

\subsection{Fano varieties and klt singularities}
\label{subsec:fano-vs-klt}
Throughout the preliminaries, we have discussed
both Fano varieties
and klt singularities.
We culminate the preliminary by presenting some 
global-to-local
and 
local-to-global principle
that enlighten the tight relation between
these global and local objects. 

Given a projective variety $X$ 
and an ample $\qq$-Cartier $\qq$-divisor $A$ on $X$, 
we can consider the {\em orbifold cone} $C_X(A)$
of $X$ with respect to the $\qq$-polarization $A$
to be the spectrum of the following ring:
\[
\bigoplus_{m\in \zz}H^0(X,\mathcal{O}_X(mA)), 
\]
where we multiply the sections as elements of $\kk(X)$.
The affine variety $C_X(A)$ is endowed with a $\mathbb{G}_m$-action 
and has a unique fixed point which we usually call
the vertex of the action.
We denote the vertex $x_0$ if there is no room for confusion.
The singularities obtained in the previous way are called
{\em cone singularities}. 
Let $\pi\colon Y\rightarrow C_X(A)$ be the blow-up of the maximal ideal at $x_0$.
Then $\pi$ extracts a unique prime divisor $E$ over $x_0$
which is isomorphic to $X$. 
If we perform adjunction of $K_X+E$ to $E$, under this isomorphism, we obtain
\begin{equation}\label{eq:adj-cone}
K_X+E|_E \sim_\qq K_X + \sum_{P \subset X} \left(1-\frac{1}{n_P}\right)P
\end{equation}
where the sum runs over all the prime divisors $P$ of $X$
and $n_P$ is the Weil index of $A$ at $P$. 
In what follows, we write $\Delta_A$ for the boundary in~\eqref{eq:adj-cone}. 
In particular, if $A$ is a Cartier divisor, then we have 
\[
K_X+E|_E \sim_\qq K_X.
\]
The previous adjunction procedure allows comparing the singularities of 
$X$ with those of $C_X(A)$.
The following result is proved in~\cite{Kol13}. 

\begin{proposition}\label{prop:from-Fano-to-klt}
Let $X$ be a projective variety and $A$ be an ample $\qq$-Cartier $\qq$-divisor on $X$. Then, the following statements hold:
\begin{enumerate}
    \item the singularity $(C_X(A);x_0)$ is klt if and only if 
    $(X,\Delta_A)$ is log Fano, 
    \item the singularity $(C_X(A);x_0)$ is canonical if and only if
    $(X,\Delta_A)$ is log Fano and $-(K_X+\Delta_A)\sim_\qq rA$ for $r\geq 1$, and
    \item the singularity $(C_X(A);x_0)$ is terminal if and only if
    $(X,\Delta_A)$ is log Fano and $-(K_X+\Delta_A)\sim_\qq rA$ for $r>1$.
\end{enumerate}
\end{proposition}

The previous procedure allows constructing a singularity
$C_X(A)$ out of a projective variety $X$ and a $\qq$-polarization.
In general, we may obtain the same singularity $(C_X(A);x_0)$
for different choices of $X$ and the polarization.
However, if we consider cone singularities with the endowed $\mathbb{G}_m$-action, then two cone singularities $(C_X(A);x_0)$ and $(C_Y(A_Y);y_0)$
are equivariantly isomorphic if and only if 
there is an isomorphism $\phi\colon X \rightarrow Y$ for which
$\phi^*A_Y=A$.
At any rate, the procedure of producing a klt type singularity out
of a Fano variety $X$ is not unique as we may choose several different ample divisors on $X$.
Anyways, the cone construction is a powerful tool to reduce projective problems to local problems.

In the other direction, 
given a klt singularity $(X;x)$. 
One possible way to associate a projective variety to $(X;x)$
is to find a strictly log canonical complement $(X,B;x)$
and take a dlt modification $(Y,B_Y)\rightarrow (X,B;x)$. 
Then, the exceptional divisors that map to $x$ are projective varieties.
However, there are some issues with this construction.
First, we may extract several divisors over the singularity $x\in X$.
Second, these divisors may not be Fano.
However, we know that the fibers of the dlt modification
(more generally, every Fano type morphism) are rationally connected
(see, e.g.,~\cite{HM07}).
In this situation, the solution is to run a suitable
minimal model program on $Y$ over $X$ 
that will contract all divisors except a single one.
This leads to the following proposition proved
by Prokhorov and Xu (see, e.g.,~\cite{Pro01,Xu14}).

\begin{proposition}\label{prop:from-klt-to-Fano}
Let $(X,B;x)$ be a klt singularity.
There exists a projective birational morphism
$\pi\colon Y\rightarrow X$ satisfying the following conditions:
\begin{itemize}
    \item the exceptional locus of $\pi$ consists of a unique prime divisor $E$ mapping onto $x$, 
    \item the divisor $-E$ is ample over $X$, and 
    \item the pair $(Y,B_Y+E)$ has plt singularities,
\end{itemize}
where $B_Y$ is the strict transform of $B$ on $Y$.
\end{proposition}

The projective birational morphism $\pi\colon Y\rightarrow X$ 
as in the previous proposition
is called a {\em purely log terminal blow-up}
or {\em plt blow-up} for short.
Hence, the previous statement says that every klt singularity
admits a plt blow-up.
The plt blow-up may not be unique.
For instance, for a toric singularity $(T;t)$, 
every toric projective morphism $Y\rightarrow T$ that extracts
a unique toric divisor over $t$ is a plt blow-up.
Plt blow-ups often help to reduce problems about
klt singularities
to problems about Fano varieties.

Both constructions; the cone construction and plt blow-ups, 
are in some sense inverse to each other.
On one side, if $X$ is a Fano variety and $A$ is an ample Cartier divisor, then $(C_X(A);x_0)$ is a klt singularity
and the blow-up of $x_0$ is a plt blow-up
whose exceptional divisor is isomorphic to $X$.
On the other hand, 
if $(X;x)$ is a klt singularity
and $Y\rightarrow X$ a plt blow-up extracting 
an exceptional divisor $E$, then 
the singularity $(X;x)$
degenerates to an orbifold cone over $E$.
In general, it is not true that $(X;x)$ is itself an orbifold cone singularity.

\section{The Coregularity} 

In this section, we turn to introduce the main object of this note: the coregularity.
The coregularity measures the difference between 
the dimension of $X$
and the dimension of the largest dual complex among log Calabi--Yau structures of $X$.
For simplicity, we first define its counterpart, the regularity.

\begin{definition}
{\em 
Let $X$ be a normal projective variety.
The {\em regularity} of $X$,
denoted by ${\rm reg}(X)$, is defined to be
\[
\max \left\{ 
\dim \mathcal{D}(X,B) \mid
\text{$(X,B)$ is log Calabi--Yau}
\right\}.  
\]
If the previous set is empty, i.e., if $X$ is not 
of log Calabi--Yau type, 
then we say that the regularity is infinite.
We set the dimension of the empty set to be $-1$.
Note that ${\rm reg}(X)=-1$ if and only if
$X$ admits a log Calabi-Yau structure
and every log Calabi-Yau structure is klt.
In other words, $X$ is exceptional.
If the regularity is a positive integer
$X$ admits a log Calabi--Yau structure
$(X,B)$ which admits a log canonical center.

The coregularity of $X$ is defined to be 
\[
{\rm coreg}(X):=
\dim X - {\rm reg}(X) -1.
\]
The 
coregularity of $X$ is 
negative infinite
if $X$ does not admit a log Calabi--Yau structure.
If the coregularity of $X$ is non-negative, then
it equals the dimension of the smallest 
log canonical center 
among dlt modifications of log Calabi--Yau pairs $(X,B)$.
If $(X,B)$ is a log Calabi--Yau pair, then we set
$X$ to be a log canonical center of the pair $(X,B)$.
Even if this seems unnatural, it does fit the adjunction formula. 

Analogously, we may define the 
{\em local regularity}
of a normal projective variety
$X$ at a closed point $x\in X$.
The local regularity, denoted by
${\rm reg}(X;x)$, is defined to be
\[
{\rm max}\{\dim\mathcal{D}(X,B;x)\mid\text{$(X,B;x)$ is log canonical}\}. 
\]
If the previous set is empty, i.e., 
if $(X;x)$ is not of log canonical type, 
then we set the local regularity to be infinite. 
Otherwise, $(X;x)$ admits a log canonical singularity structure.
If $(X;x)$ is a klt type singularity, then
the regularity is always at least zero. 
Again, the coregularity is the dimension
of the smallest minimal log canonical center
among the dlt modifications of lc singularities $(X,B;x)$.
}
\end{definition}

In~\cite[Section 7]{Sho00}, Shokurov defined the regularity of a contractions.
He mentions that the regularity 
characterizes the topological difficulty of the pair.
The previous definitions can be given
in a more general context of
Fano type morphisms $X\rightarrow Z$
with a marked closed point $z\in Z$.
If $Z$ is a point, then this recovers the global definition.
If $X\rightarrow Z$ is the identity, then this recovers the local definition.
Our definition of regularity tries to maximize the dimension
among all the possible dual complexes. 
On the other hand, our definition of coregularity tries
to minimize the difference between 
$\dim X$ and ${\rm reg}(X,B)$.

As for invariants of Fano varieties, 
the most well-known invariants
are the index
and the anti-canonical volume $(-K_X)^n$. 
As explained before, 
the index measures
how divisible the anti-canonical divisor
$-K_X$ is in the Picard group of $X$.
The volume controls the asymptotic growth
of sections of multiples of the anti-canonical divisor. 
It is expected that most Fano varieties have index one.
And, although this statement can not be made precise 
in a general setting, it is often the case when 
looking at precise families of examples.
However, we do not expect any general framework 
for Fano varieties of index one.
On the other hand, there is no such expectation for volumes.
The volume, even in the toric case, 
can take many different values in a fixed dimension.
Usually, bounding the volume from below and/or 
above is related to boundedness statements in algebraic geometry.
As for the coregularity, 
we expect that most Fano varieties
have coregularity zero. 
In the following sections, we give some examples
and propose structural conjectures about
Fano varieties of coregularity zero.

\section{Examples and Properties}
This section explores examples and properties of the coregularity.
In the first part of this section, we focus on examples.
In subsection~\ref{subsec:del-Pezzo},
we discuss del Pezzo surfaces, 
while
in subsection~\ref{subsec:sing-del-Pezzo}, 
we give examples of singular del Pezzo surfaces.
In subsection~\ref{subsec:quot-sing}, we discuss finite quotient singularities
and in subsection~\ref{subsec:quotients}, we study the projective 
analogue; finite quotients of the projective space.
In subsection~\ref{subsec:ter-3-fold}, we discuss terminal 3-fold singularities,
while in subsection~\ref{subsec:klt-3-fold}, we study 
klt $3$-fold singularities in a more general setting.
In this second part of this section, we focus on properties.
We study Fano varieties with torus actions in subsection~\ref{subsec:torus-actions},
the coregularity under morphisms in subsection~\ref{subsec:coreg-morphisms}, and 
the coregularity under deformations and degenerations in
subsection~\ref{subsec:coreg-deform}.
Finally, in subsection~\ref{subsec:prop-dual-complexes} and subsection~\ref{subsec:ex-dual-complexes}, 
we give further properties and examples of dual complexes.

\subsection{Del Pezzo surfaces}
\label{subsec:del-Pezzo}
del Pezzo surfaces 
are classified by the anti-canonical volume $(-K_X)^2$
which is also called the degree 
of the del Pezzo surface.
The degree of a del Pezzo
surface is an integer 
in $\{1,2,\dots,9\}$.

There is a unique 
del Pezzo surface of degree $9$, 
the projective plane $\pp^2$.
In this case, we can take the 
$B=L_1+L_2+L_3$ 
so that 
$\mathcal{D}(\pp^2,B)$
is a circle with three marked points 
corresponding to the lines.
Then, the coregularity of $\pp^2$ is zero.
A del Pezzo surface of degree $d$ in
$\{8,7,6\}$
is the blow-up of
$\pp^2$ at 
$9-d$ points. 
For these surfaces, we may apply
an automorphism of $\pp^2$
and assume that they are 
the intersection points
of $L_1\cap L_2$, $L_1\cap L_3$ and $L_2\cap L_3$.
Hence, the dual complex of the log pull-back 
of $(\pp^2,L_1+L_2+L_3)$ to these models
are circles with $12-d$ marked points.
Indeed, the log discrepancy of the exceptional divisors is $2$ with respect to $\pp^2$.
Indeed, these are just the blow-ups of smooth points.
Moreover, each line decreases this log discrepancy exactly by one.
Hence, 
the del Pezzo surfaces of degree 
$d\in \{8,7,6\}$ have
coregularity zero 
and the number of points in the dual complex is maximal
up to isomorphism.
All the previous examples are toric.

Up to isomorphism, there is a unique
del Pezzo surface of degree $5$.
This surface can be obtained from $\pp^2$
by blowing up $4$ points
with no three of them collinear.
We may assume, up to an automorphism, 
that the three points are the intersections
$L_1\cap L_2$, $L_1\cap L_3$, and $L_2\cap L_3$. 
The log pull-back of $(\pp^2,L_1+L_2+L_3)$ to this blow-up
is not a log Calabi-Yau pair, so
we need to choose a different boundary.
We can choose a conic $C$ through the $4$ points
and a line $L$ that can be chosen to be in general position.
The log pull-back of $(\pp^2,C+L)$ to this del Pezzo
is a log Calabi-Yau pair of the form
$(X,C_X+L_X)$, where $C_X$ and $L_X$ are the strict transform
of the conic and the line respectively. 
Since we are not blowing up any intersection point
we have that $C_X\cap L_X$ consists of two points.
The dual complex
$\mathcal{D}(X,C_X+L_X)$ is a circle with two marked points, corresponding to $C_X$ and $L_X$. 
The self-intersection of $C_X$ is $-3$ and the self-intersection of $L_X$ is $1$.

A del Pezzo surface of degree $4$
is a Segre surface in $\pp^4$ defined by the intersection of two quadrics.
They form a $2$-dimensional family 
and they have sixteen $(-1)$-curves.
Del Pezzo surfaces 
of degree $4$ can be described 
as the blow-up of $\pp^2$ 
at $5$ points
with no three collinear. 
A del Pezzo surface of degree $3$
is a cubic surface of degree three in $\pp^3$.
They form a $4$-dimensional family 
and they have twenty-seven $(-1)$-curves. 
A del Pezzo surface of degree $2$
is a double cover of $\pp^2$ branching
over a quartic plane curve.
They form a $6$-dimensional family 
and they have fifty-six $(-1)$-curves.
del Pezzo surfaces of degree $d\in \{2,3,4\}$
can be described as blow-ups of $\pp^2$
at $9-d$ points with no three collinear
and no six on a conic. 
In this case, we can mimic the construction 
from the previous paragraph. 
We choose a conic $C$ through five
of the $9-d$ points
and a line $L$ passing through two of them.
Then, the log pull-back of $(\pp^2,C+L)$
is log Calabi-Yau of coregularity zero.
Hence, all del Pezzo's of degree at least two
have coregularity zero.

A del Pezzo surface of degree $1$
is a double cover of a quadratic cone in $\pp^3$
branched over a smooth genus $4$-curve.
This surface can be described
as the blow-up of 
$\pp^2$ at eight points
with no three collinear, 
no six lying on a conic, 
and no eight lying on a 
cubic with a node at one of them.
In this case, we can choose a cubic $C_3$
containing the $8$ points.
Then, the log pull-back of $(\pp^2,C_3)$
to the del Pezzo of degree $1$
is log Calabi-Yau of coregularity one.
Indeed, the dual complex is just a point
corresponding to $C_3$. 
Hence, a del Pezzo of degree 
one has coregularity at most one.

The following proposition is proved
by Prokorov and Shokurov in~\cite{Pro01,Sho00}.

\begin{proposition}
Let $X$ be a Fano type surface of coregularity zero.
Then $X$ either admits a $1$-complement
or a $2$-complement of coregularity zero.
\end{proposition}

Now, we turn to study the 
$1$-complements
and
$2$-complements
of a log del Pezzo of degree one.
Let $X$ be a log del Pezzo of degree one
and $B$ a boundary on $X$
for which
$(X,B)$ is log Calabi-Yau and $2(K_X+B)\sim 0$.
We have a projective birational morphism
$\pi\colon X\rightarrow \pp^2$
which is the blow-up at $8$ points.
The push-forward of $K_X+B$ to $\pp^2$
induces a boundary $B_2$ on $\pp^2$
for which $2(K_{\pp^2}+B_2)\sim 0$
and $\pi^*(K_{\pp^2}+B_2)=K_X+B$.
Note that each of the eight points 
must be contained in the support of $B_2$.

We study the possible structures
that the divisor $B_2$ can have. 
By counting the degree, we can write
\[
B_2 = \frac{1}{2} \left( \sum_{i=1}^k C_i \right),
\]
where $k$ is at most $6$.
We analyze the case in which $k=2$.
In this case, we have a log Calabi-Yau pair
\begin{equation}\label{eq:two-comp-pair}
\left(\pp^2, \frac{1}{2}C_a + \frac{1}{2}C_b \right),
\end{equation} 
where $C_a$ is a curve of degree $a$,
$C_b$ is a curve of degree $b$,
and $a+b=6$.
Note that $C_1+C_2$ must pass with multiplicity
at least two at each of the eight points of the blow-up
$X\rightarrow \pp^2$.
If the pair 
$(\pp^2,(C_a+C_b)/2)$ has coregularity one, 
then $C_a+C_b$ must pass through one point with multiplicity
at least four.
If the pair
$(\pp^2,(C_a+C_b)/2)$ has coregularity zero, 
then $C_a+C_b$ must pass through one point with multiplicity at least six.
The possible pairs for
the dimension of the spaces of curves of degree $a$ and $b$
are $(28,0)$,$(21,3)$,$(15,6)$, and $(10,10)$, up to permutation.
Note that the seven double points
impose
$21$ conditions on these spaces
while the triple point
imposes $6$ conditions.
In total, we have $27$ independent conditions imposed
on the space of curves. 
This implies that no pair of the form~\eqref{eq:two-comp-pair}
can have coregularity zero.
A similar analysis concludes the following theorem.

\begin{theorem}
A del Pezzo surface has coregularity zero
if and only if it has degree at least two.
\end{theorem}

In summary, del Pezzo surfaces of coregularity one 
form a $8$-dimensional family.
On the other hand, del Pezzo surfaces of coregularity zero
are the union of
a $2$-dimensional family, 
a $4$-dimensional family, 
a $6$-dimensional family, 
and $6$ other examples, from which four are toric.
No del Pezzo surface has coregularity two. 

There has been a lot of work on classifying
Gorenstein del Pezzo surfaces, 
i.e., surfaces $X$ with $-K_X$ ample
and canonical singularities. 
These surfaces are classified
by the fundamental group
of its smooth locus (which is finite)
in the work of Miyanishi and Zhang
(see, e.g.,~\cite{MZ88,MZ93}).
It would be interesting to extend the previous
characterization 
of del Pezzo surfaces of coregularity zero 
to Gorenstein del Pezzo.

\begin{problem}
Classify Gorenstein del Pezzo surfaces
of coregularity zero, one, and two.
\end{problem}

Using the work of Prokhorov and Shokurov~\cite{Pro01,Sho00}, 
one can prove the following result 
which characterizes Fano type surfaces of coregularity zero.

\begin{theorem}
Let $X$ be a Fano type surface of coregularity zero.
Then there exists a $2$-complement $(X,B)$ 
and a cover $\pi\colon Y\rightarrow X$ of degree at most two,
such that the log pull-back $(Y,B_Y)$ of $(X,B)$ to $Y$
is log crepant to $(\pp^2,L_1+L_2+L_3)$.
\end{theorem}

In the previous theorem, 
we say that $(Y,B_Y)$ is log crepant to $(\pp^2,L_1+L_2+L_3)$
if there exists a common resolution
of singularities
$p\colon Z\rightarrow Y$
and $q\colon Z\rightarrow \pp^2$ for which we have 
\[
p^*(K_Y+B_Y) = q^*(K_{\pp^2}+L_1+L_2+L_3).
\] 
Furthermore, in the previous case we can go
from $(\pp^2,L_1+L_2+L_3)$
to $(Y,B_Y)$ by blowing up a sequence of points
in the total transform of $L_1+L_2+L_3$
and blowing down a sequence of curves.

The following examples shows
that complements
computing the coregularity
can have arbitrarily small coefficients.

\begin{example}
\label{ex:example-coreg-0}
{\em 
Consider $n$ conics 
$C_1,\dots, C_n$ in $\pp^2$ 
passing smoothly through the origin $[0:0:1]$
with common tangent direction.
\begin{figure} 
\includegraphics[scale=0.4]{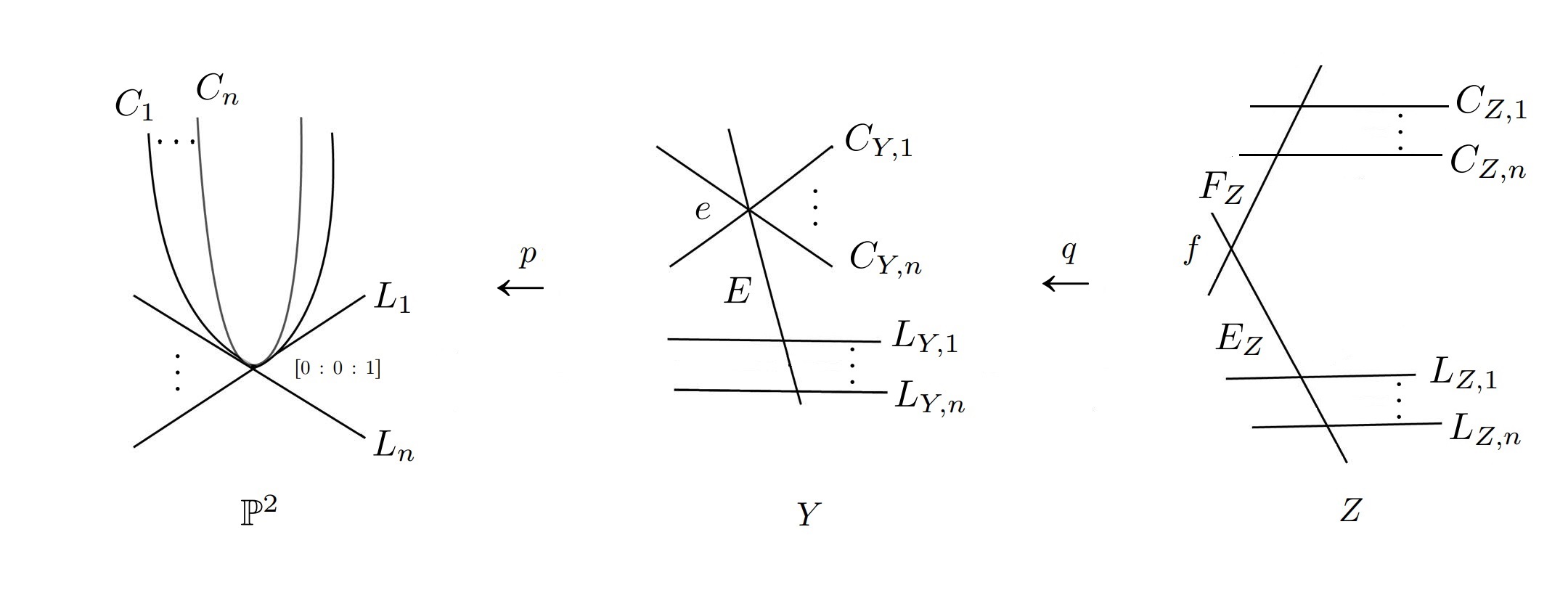}
\caption{Sequence of blow-ups.}
\label{fig0}
\end{figure} 
See Figure~\ref{fig0}.
Let $L_1,\dots,L_n$ be general lines
passing through the origin
with general tangent directions.
Then, the pair
\begin{equation}\label{eq:log-pair-small-coeff} 
\left(\pp^2, \frac{1}{n}(L_1+\dots+L_n)+\frac{1}{n}(C_1+\dots+C_n)\right)
\end{equation}  
is log Calabi-Yau.
We claim that it has coregularity zero.
For simplicity we let 
$L:=(L_1+\dots+L_n)/n$
and
$C:=(C_1+\dots+C_n)/n$.
First, we show that the origin
is a log canonical center
and more precisely the exceptional extracted by blowing-up its maximal ideal
is a log canonical place.
Indeed, each line and conic pass through the point with multiplicity one, so the sum $(L_1+\dots+L_n+C_1+\dots+C_n)/n$
pass through the origin with multiplicity two. 
Let $p\colon Y\rightarrow \pp^2$ be the blow-up at the origin.
Then, we have that 
\[
p^*\left(K_{\pp^2}+L+C\right)=
K_{Y}+L_{Y}+C_{Y}+E,
\] 
where $E$ is the exceptional divisor,
$L_{Y}$ (resp. $C_{Y}$) 
the strict trasnform of $L$ (resp. $C$).
Analogously, each $L_{Y,i}$ (resp. $C_{Y,i})$ is the strict transform of $L_i$
(resp. $C_i$).
Note that the intersection
of each component of $L_{Y}$ and $E$ is transversal.
The intersection of each component of $C_{Y}$ with $E$ is transversal.
The intersection of all the components
of $C_{Y}$ is a point $e\in E$.
Let $q\colon Z\rightarrow Y$ be the blow-up of $Y$ at $e$. 
Then, we have that 
\[
q^*(K_Y+L_{Y}+C_{Y}+E)=
K_Z+L_{Z}+C_{Z}+E_Z+F_Z 
\] 
where $E_Z$ is the strict transform of $E$ on $Y$ and $F_Z$ is the exceptional of $q$.
Analogously, we denote by
$L_{Z,i}$ (resp. $C_{Z,i}$) the strict transform of
$L_{Y,i}$ (resp. $C_{Y,i}$) on $Z$.
Note that $f=E_Z\cap F_Z\neq \emptyset$.
This model is actually a log resolution as all components
of $L_{Z},C_{Z},E_Z$ and $F_Z$ intersect transversally. 
Hence, the pair~\eqref{eq:log-pair-small-coeff} has coregularity zero.
}
\end{example}

\subsection{Surface quotient singularities}
\label{subsec:quot-sing}
In dimension two, klt singularities are finite quotient singularities. 
The most well-known of these are the DuVal singularities 
described by the following equations:
\begin{align*}
    A_n: \quad &  x^2+y^2+z^{n+1} =0,\\
    D_n: \quad &  x^2+y^2z+z^{n-1} =0, \quad n\geq 4.\\
    E_6: \quad &  x^2+y^3+z^4=0, \\
    E_7: \quad &  x^2+y^3+yz^3=0, \\
    E_8: \quad &  x^2+y^3+z^5=0.
\end{align*}
The $A_n$ singularity is toric, indeed it can also be written as
$xy+z^{n+1}=0$ 
and the pair
$(A_n,\{x=0\}+\{y=0\})$
has coregularity zero. 
Indeed, every curve in the minimal resolution of the $A_n$
appear with coefficient one 
in the log pull-back of the previous pair.
Thus, the dual complex in the minimal resolution
is a closed segment with $n+2$ marked points,
from which $n$ correspond to exceptional curves
and $2$ correspond to $\{x=0\}$ and $\{y=0\}$.
Hence, the $A_n$ singularity has coregularity zero.

The $D_n$ singularity
admits a $2$-complement given
by the pair
$(D_n, \{y=z=0\})$.
The strict transform of
the curve
$\{y=z=0\}$ to the minimal resolution 
intersects a single curve
which is the endpoint of one of the branches
of the curve that has length different than two.
The index one cover
of $K_{D_n}+\{y=z=0\}$ is a $A_n$
singularity with its torus invariant boundary.
Hence, the $D_n$ singularity 
has coregularity zero.

The singularities 
$E_6,E_7$ and $E_8$
are called exceptional singularities. 
They have coregularity one
and the only divisor over
them which can compute a log canonical place
is the center of the fork.
Nevertheless, these singularities
admit $6$-complements.
From now on, we will say that a singularity
is {\em exceptional} if its coregularity is its dimension minus one.
If a klt singularity is exceptional, 
then we can perform a blow-up which has a unique
prime exceptional divisor that is an exceptional Fano variety~\cite{Mor18c,HLM20,Mor21a}.
Exceptional Fano varieties of dimension $n$ are in a bounded family
by the work of Birkar~\cite{Bir19}.
Exceptional singularities
of dimension $n$ can be deformed into
orbifold cones over exceptional Fano varieties. 
This leads to the following theorem due to Han, Liu, and the author~\cite{HLM20}:

\begin{theorem}
Let $n$ be a positive integer
and $\epsilon>0$ be a positive real number.
The class of $n$-dimensional exceptional singularities
$(X;x)$ with log discrepancy at least $\epsilon$ 
are bounded up to deformation.
\end{theorem}

In the previous statement 
bounded up to deformation means that 
there exists a bounded set of singularities $\mathcal{B}$
so that every element as in the statement deforms 
to an element in $\mathcal{B}$.
Due to the previous theorem, 
we do not expect interesting behaviours
on exceptional singularities. 
Most invariants take only finitely many possible values
on $n$-dimensional exceptional singularities
with log discrepancy at least $\epsilon>0$.
The following theorem 
characterizes klt
surface singularities of coregularity zero.

\begin{theorem}\label{thm:surf-coreg-0}
Let $(X;x)$ be a $2$-dimensional klt singularity
of coregularity zero.
Then, there exists a reduced $2$-complement
$(X,B;x)$ for which the index one cover
of $K_X+B$ is toric around the inverse image of $x$.
\end{theorem}

\begin{proof}
By definition, we can find a $\qq$-complement
$\Gamma$ so that the singularity
$(X,\Gamma;x)$ has coregularity zero. 
Let $\pi\colon Y\rightarrow X$ be a dlt modification
of $(X,\Gamma;x)$.
Note that, since $(X;x)$ is klt, then $\Gamma$ is non-trivial at $x$.
Write 
\[
\pi^*(K_X+\Gamma) = K_Y+S_1+\dots+S_r+\Gamma_Y,
\]
where $\Gamma_Y$ is the strict transform of $\Gamma$ on $Y$
and the $S_i$'s are the exceptional divisors.
By~\cite[Corollary 7.12]{Sho00}, the dual complex
$\mathcal{D}(S_1+\dots+S_r)$ is a segment of a line.
We assume $S_1$ and $S_r$ are its endpoints.
Let 
$(S_1, S_2|_{S_1} + B_{S_1})$
and
$(S_r, S_{r-1}|_{S_r} + B_{S_r})$
be the log pairs
induced by adjunction of
$(K_Y,S_1+\dots+S_r)$ to $S_1$ and $S_r$, respectively.
Note that the coefficients of $B_{S_1}$ and $B_{S_r}$ are standard. 
Hence, we may assume that one of the following situations hold:
\begin{itemize}
    \item Each boundary divisor 
    $B_{S_1}$ and $B_{S_r}$ have a unique prime component
    with coefficients $1-\frac{1}{m_1}$ and $1-\frac{1}{m_r}$, or
    \item the divisor $B_{S_1}$ has a unique prime component
    and $B_{S_r}$ has two components of coefficient $\frac{1}{2}$.
\end{itemize}
In the former case, we can find curves $B_1$ and $B_r$ with coefficient one on 
$Y$ such that $B_1$ (resp. $B_r$) intersects 
$S_1+\dots+S_r$ at $B_{S_1}$ (resp. $B_{S_r}$) with multiplicity one.
Then, the pair 
$(Y,S_1+\dots+S_r+B_1+B_r)$ is log Calabi--Yau over $X$.
We set $B$ to be the push-forward of $B_1+B_r$ to $X$.
Then, $(X,B)$ is a toric singularity by Theorem~\ref{thm:complexity-toric} below.
Indeed, here the two irreducible components of $B$
play the role of the torus invariant divisors for the action.

In the latter case, we can find a curve $B_1$ which intersects
$S_1+\dots+S_r$ at $B_{S_1}$ with multiplicity one. 
Then, the log pair
$(Y,S_1+\dots+S_r+B_1)$ is log Calabi--Yau.
Let $B$ be the push-forward of $B_1$ to $X$.
Then, the pair $(X,B;x)$ is log canonical.
Indeed, its log pull-back to $Y$ is simply 
$(Y,S_1+\dots+S_r+B_1)$ which has log canonical singularities.
Let $X'\rightarrow X$ be the index one cover of $K_X+B$.
By construction, we have that $2(K_X+B)\sim 0$. 
Indeed, we have that 
\[
2(K_Y+S_1+\dots+S_r+B_1)\sim_X 0
\]
and this is preserved under push-forwards.
Hence, the index one cover of $K_X+B$
is a two-to-one cover $X'\rightarrow X$
that is unramified outside the point $x\in X$.
Indeed, $B$ is reduced.
Let $\phi\colon Y'\rightarrow Y$ be the normalization
of the main component of the fiber product
$X'\times_X Y$.
Then, we can write
\[
\phi^*(K_Y+S_1+\dots+S_r+B_1)=
K_{Y'}+E_1+\dots+E_r+\dots+E_{2r-1}+B'_1+B'_2.
\]
The involution acts on 
$E_1+\dots+E_r+\dots+E_{2r-1}$
by swapping $E_i$ with $E_{r-i}$ for $i\neq r$
and on $E_r$ the involution sends $e$ to $e^{-1}$.
The boundary $B'_1+B'_2$ is simply the pull-back of
$B_1$ to $Y'$ and each component intersects 
$E_1$ and $E_{2r-1}$, respectively.
Let $B'$ be the push-forward of $B'_1+B'_2$ to $X'$.
Then, the pair $(X',B';x')$ is log canonical.
The same argument as in the first case implies
that $(X',B';x')$ is a toric singularity.
\end{proof}

\subsection{Finite quotients of $\mathbb{P}^n$}\label{subsec:quotients}
As discussed in the introduction,
quotients of the projective space
by finite groups are Fano varieties.
Often, these have singularities coming from the fixed
points of the action. 
Analogously, we can consider the local picture, 
i.e., quotients of the affine space (or a smooth germ)
by the action of a finite group fixing the origin.
In the same vein, 
we can consider quotients of del Pezzo surfaces 
by finite group actions.

The quotient of $\pp^1$ by the action of a
finite group $G$ is exceptional if and only if
$G$ is either the $E_6,E_7$ or $E_8$ group.
Here, we are considering the quotient with the 
usual logarithmic structure endowed by
the Riemann-Hurwitz formula.
We recall that by the Chevalley-Shephard-Todd theorem, 
the quotient of $\cc^n$ by a finite group
$G$ generated by pseudo-reflections is isomorphic to $\cc^n$.
Hence, when discussing quotient singularities,
we may always assume that $G$ does not contain pseudo-reflections.
In~\cite{MP99}, Markushevich and Prokhorov
proved that a surface quotient singularity
$\cc^2/G$ is exceptional 
if and only if $G$ 
has no semi-invariants of degree at most $2$.
Similarly, the quotient 
$\cc^3/G$ is exceptional if and only if $G$
has no semi-invariants of degree at most $3$.

The del Pezzo surface $X_5$ of degree $5$
has an automorphism group isomorphic to 
$S_5$.
The quotient of $X_5$ by 
$A_5$ and $S_5$ are exceptional
(see, e.g.,~\cite{Che08}).
Using the Miller–Blichfeldt–Dickson
classification
of finite subgroups of $\mathbb{P}{\rm GL}_3(\cc)$, 
the exceptional Fanos obtained
as finite quotients
$\pp^2/G$ can be classified.
Indeed, all these exceptional Fanos are 
quotients by an element in one of the four imprimitive
classes $F,G,I,J$.
This statement has a local analog for
canonical exceptional singularities.
In~\cite{MP99}, the authors prove that 
a quotient singularity $\cc^3/G$ is 
a canonical exceptional singularity
if and only if $G$ is one of the following groups:
\begin{enumerate}
    \item The Klein's simple group,
    \item the unique central extension of the Klein simple group
    in ${\rm SL}_3(\cc)$, 
    \item the Hessian group,
    \item the normal subgroup of the Hessian group, or
    \item a central extension of $\mathcal{U}_6$.
\end{enumerate}

In~\cite{CS11}, the authors study
exceptional quotients $\pp^n/G$ for higher-dimensional projective spaces.
In dimension $3$, Cheltsov shows that the quotient is exceptional
provided that $|G|\geq 169$ and $G$ does not
have semi-invariants of degree at most $4$.
There is further work on the exceptionality of 
quotients $\pp^n/G$ in dimension at most $9$ by
Cheltsov and Shramov~\cite{CS11b,CS12}.
We propose a problem in the opposite direction:

\begin{problem}
Characterize, in terms of semi-invariants, 
the quotients 
$\mathbb{A}^n/G$ 
and 
$\mathbb{P}^n/G$ that have coregularity zero.
\end{problem}

The expected answer is that, 
whenever $G$ has enough semi-invariants
of small degree, then the quotient should have
coregularity zero. 
For instance, if 
the group $G$ fixes the sum 
of the hyperplanes
$H_1+\dots+H_n$, then the 
quotient $\mathbb{A}^n/G$ has coregularity zero.
However, this is an easy example. 
It would be interesting to have complete classifications
even in low dimensions.
In general, we have the following inequality 
between coregularities of quotients:

\begin{proposition}
Let $(X;x)$ be a klt type singularity 
and $G$ be a finite group acting on it.
Let $Y=X/G$ and $y$ be the image of $x$.
Write $p\colon X\rightarrow Y$ for the quotient
and $p^*(K_Y+\Delta_Y)=K_X$ for some divisor
$\Delta_Y$ with standard coefficients.
Then, we have that
\[
{\rm coreg}(X;x) \leq {\rm coreg}(Y,\Delta_Y;y).
\]
\end{proposition}

\begin{proof}
Consider a complement
$\Gamma_Y$ on $(Y,\Delta_Y;y)$
which computes the coregularity at $y$.
Then, the pair
$(Y,\Delta_Y+\Gamma_Y;y)$ is log canonical
and $y\in Y$ is a log canonical center.
Then, the pull-back
\[
p^*(K_Y+\Delta_Y+\Gamma_Y)=K_X+\Gamma_X,
\]
defines a log canonical pair
$(X,\Gamma_X;x)$
such that $x$ is a log canonical center. 
The dual complex
$\mathcal{D}(X,\Gamma_;x)$
quotients
to the dual complex
$\mathcal{D}(Y,\Delta_Y+\Gamma_Y;y)$
by the induced action of $G$.
In particular, both dual complexes
has the same dimension. 
This implies that 
\[
{\rm coreg}(Y,\Delta_Y;y) = 
{\rm coreg}(Y,\Delta_Y+\Gamma_Y;y) = 
{\rm coreg}(X,\Delta_X;x) \geq
{\rm coreg}(X;x).
\]
This finishes the proof.
\end{proof}
We note that the previous inequality
is not sharp.
Indeed, for exceptional quotient singularities
of dimension $n$ 
the affine space $\cc^n$ has coregularity zero
while the quotient has coregularity $n-1$.
Thus, the difference between these coregularities
can be as large as possible.
A more ambitious question in this direction
is to try to determine for which finite actions
$G$ on klt singularities $(X;x)$
the quotient $(X/G,\bar{x})$ has the same coregularity as $(X;x)$.

\subsection{Singular del Pezzo surfaces}
\label{subsec:sing-del-Pezzo}
In this subsection, we discuss complements
and coregularity
of singular del Pezzo surfaces, i.e., 
surfaces $X$ with klt singularities
for which $-K_X$ is ample. 
We will focus on surfaces of Picard rank one. 
Before proceeding to construct examples, we will review the 
ADE classification of quotient singularities. 

As explained above, Du Val singularities
are classified by the ADE classification, i.e., they are either
$A_n$, $D_n$, $E_6,E_7$ or $E_8$ singularities. 
This classification generalizes to 
klt surface singularities. 

\begin{definition}
{\em 
Let $(X;x)$ be a klt surface singularity.
We say that $(X;x)$ is of {\em $A$-type} 
if the graph of its minimal resolution is an interval with vertices.
We say that $(X;x)$ is of {\em $D$-type}
if the graph of its minimal resolution is a fork with three branches 
and two branches have length one. 
In the $D$-type case, we assume that the third branch is non-empty.
We say that $(X;x)$ is of {\em $E$-type}
if it is not of $A$-type or $D$-type.
}
\end{definition} 

It follows from the work of Alexeev that every klt surface singularity belongs to one of the previous classes (see, e.g.,~\cite{Ale93}).
We have the following proposition (see, e.g.,~\cite{Sho00}).

\begin{proposition}\label{prop:comp-klt-surf}
Let $(X;x)$ be a klt surface singularity.
Then, one of the following statements hold:
\begin{enumerate}
    \item {\rm $A$-type:} there exists a $1$-complement $(X,B;x)$ and $(X,B)$ is formally toric around $x$, 
    \item {\rm $D$-type:} there exists a reduced $2$-complement $(X,B;x)$ and the index one cover of $K_X+B$ is formally toric at the pre-image of $x$, or 
    \item {\rm $E$-type:} there exists a $6$-complement $(X,B;x)$ and the dual complex $\mathcal{D}(X,B;x)$ is a point.
\end{enumerate}
Furthermore, in the case of $D$-type singularities there is no non-trivial $1$-complement
through $x\in X$.
\end{proposition}

\begin{proof}
The statements $(1),(2)$, and $(3)$ follow from the theory of complements for surfaces~\cite{Sho00}. 
We argue that a $D$-type singularity admits no $1$-complement.
Let $(X;x)$ be a $D$-type singularity. 
Let $(X,B;x)$ be a $1$-complement. 
Assume that $B$ is non-trivial, i.e., it passes through $x$.
Then, every divisor of the minimal resolution of $(X;x)$ has log discrepancy zero with respect to $(X,B;x)$. 
Let $(Y,B_Y)$ be a dlt modification of $(X,B;x)$. 
Then, $B_Y$ contains a curve $C$ which intersects transversally three other curves of $\lfloor B_Y\rfloor$ ($C$ is the fork point in the minimal resolution).
Then, the pair $(C,B_C)$ obtained from adjunction of $K_Y+B_Y$ to $C$
satisfies that $B_C$ has three points of coefficient one. This leads to a contradiction.
We conclude that $(X;x)$ admits no $1$-complement.
\end{proof}

The previous proposition states that the only difference between klt surface singularities
and Du Val singularities
are the self-intersections
of the curves
in the minimal resolution.
Furthermore, if two singularities
have the same resolution
and self-intersections of the curves
in the minimal resolution,
then they are analytically isomorphic.
The $A$-type
and $D$-type singularities
have coregularity zero
while $E$-type singularities
have coregularity one.

We turn to study some examples
of complements
on Gorenstein del Pezzo surfaces. 
We recall that 
Gorenstein del Pezzo surfaces
of Picard rank one
are classified by the work 
of Miyanishi and Zhang (see, e.g.,~\cite{MZ88}).
The following example
is a Gorenstein del Pezzo surface of coregularity zero that admits
a $1$-complement and a $2$-complement,
both of coregularity zero.

\begin{example}
\label{ex:a3d5}
{\rm 
Let $X$ be a Gorenstein del Pezzo surface
of Picard rank one
with two singular points of type
$A_3$ and $D_5$.
We write $x$ and $y$ for these points, respectively.
Let $V\rightarrow X$ be the minimal resolution of $X$.
\begin{figure} 
\includegraphics[scale=0.2]{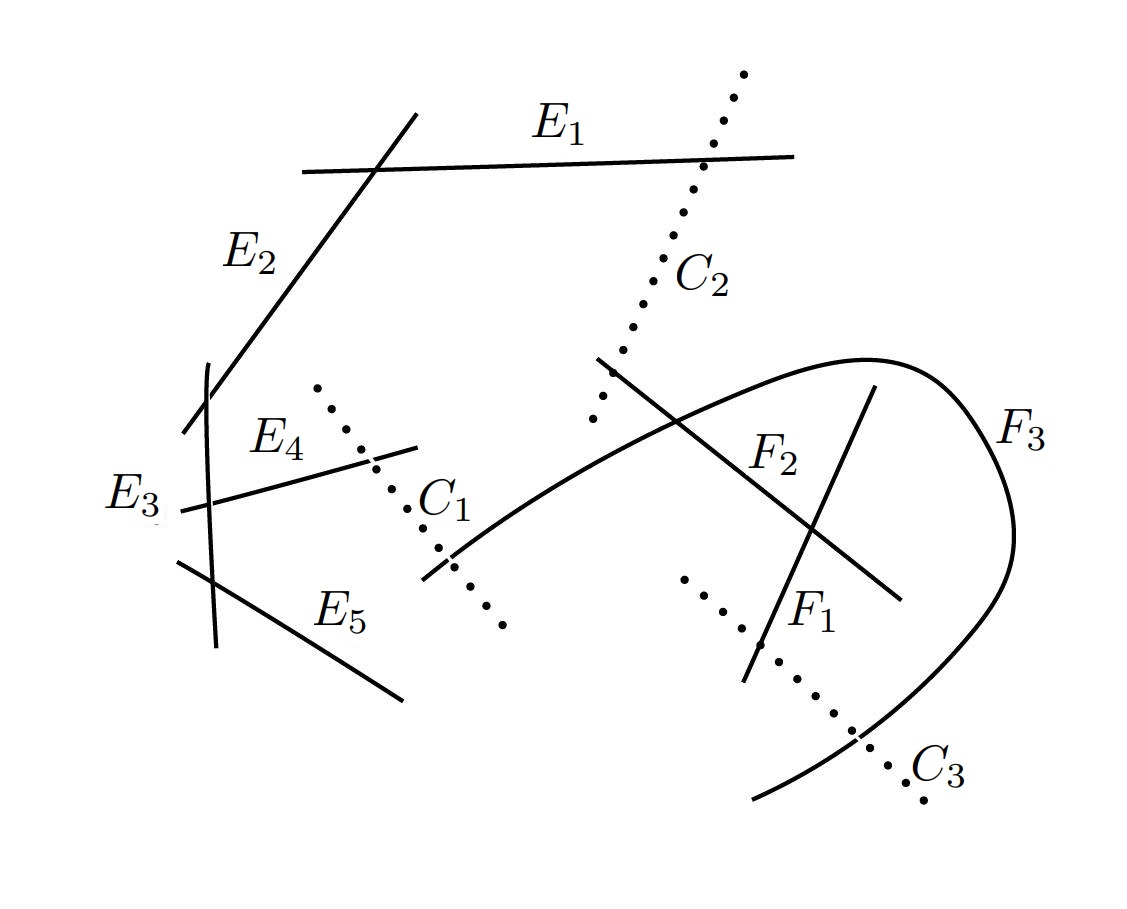}
\caption{$A_3$ + $D_5$ singularities.}
\label{fig1}
\end{figure} 
In Figure~\ref{fig1}, the solid lines correspond to the $(-2)$-curves
extracted on the minimal resolution,
while the dotted lines
are strict transforms of curves 
on $X$ which become $(-1)$-curves on $V$.
The curves $E_1,\dots,E_5$ are the exceptional divisors of the resolution of the $D_5$ singularity.
The curves
$F_1,F_2$, and $F_3$ are the exceptional divisors of the resolution of the $A_3$ singularity.
The curves $C_1,C_2$, and $C_3$ are strict transform of curves on $X$.
We can contract the $(-1)$-curves on $V$ 
repeatedly 
until we obtain a Hirzebruch surface $\Sigma_2$. 
There are some interesting complements on $X$. 

First, we can consider the reduced $2$-complement $(X,C_2)$.
Locally around the $D_5$-singularity $(X;y)$, this complement is just the standard reduced $2$-complement $(X,C_2;x)$.
On the other hand, around the $A_3$-singularity $(X;x)$, this complement
is a $2$-complement which corresponds to the usual $2$-complement when 
considering this singularity as a $D_3$-singularity.
The dual complex $\mathcal{D}(X,C_2)$
is an interval with five vertices, corresponding to $E_3,E_2,E_1,C_2,$ and $F_2$.
Thus, $X$ has coregularity zero.

On the other hand, we can consider the 
$1$-complement $(X,C_3)$. 
Note that $C_3$ is disjoint from the $D_5$-singularity $y$.
On the other hand, the dual complex
$\mathcal{D}(X,C_3)$
is a circle with four vertices corresponding to $F_1,F_2,F_3,C_3$.
We conclude that $X$ is a Gorenstein del Pezzo surface of Picard rank one 
and coregularity zero.
Furthermore, it admits both a $1$-complement
and a $2$-complement that compute the coregularity.
}
\end{example}

The following is an example of a 
Gorenstein del Pezzo surface 
of coregularity zero.
It admits a $2$-complement of coregularity zero and every $1$-complement on it has coregularity one.

\begin{example}
\label{ex:d4d4}
{\rm 
Let $X$ be a Gorenstein del Pezzo surface of Picard rank one with two singular points of type
$D_4$.
We will write $x$ and $y$ for these points.
Let $V\rightarrow X$ be the minimal resolution of $X$.
\begin{figure} 
\includegraphics[scale=0.2]{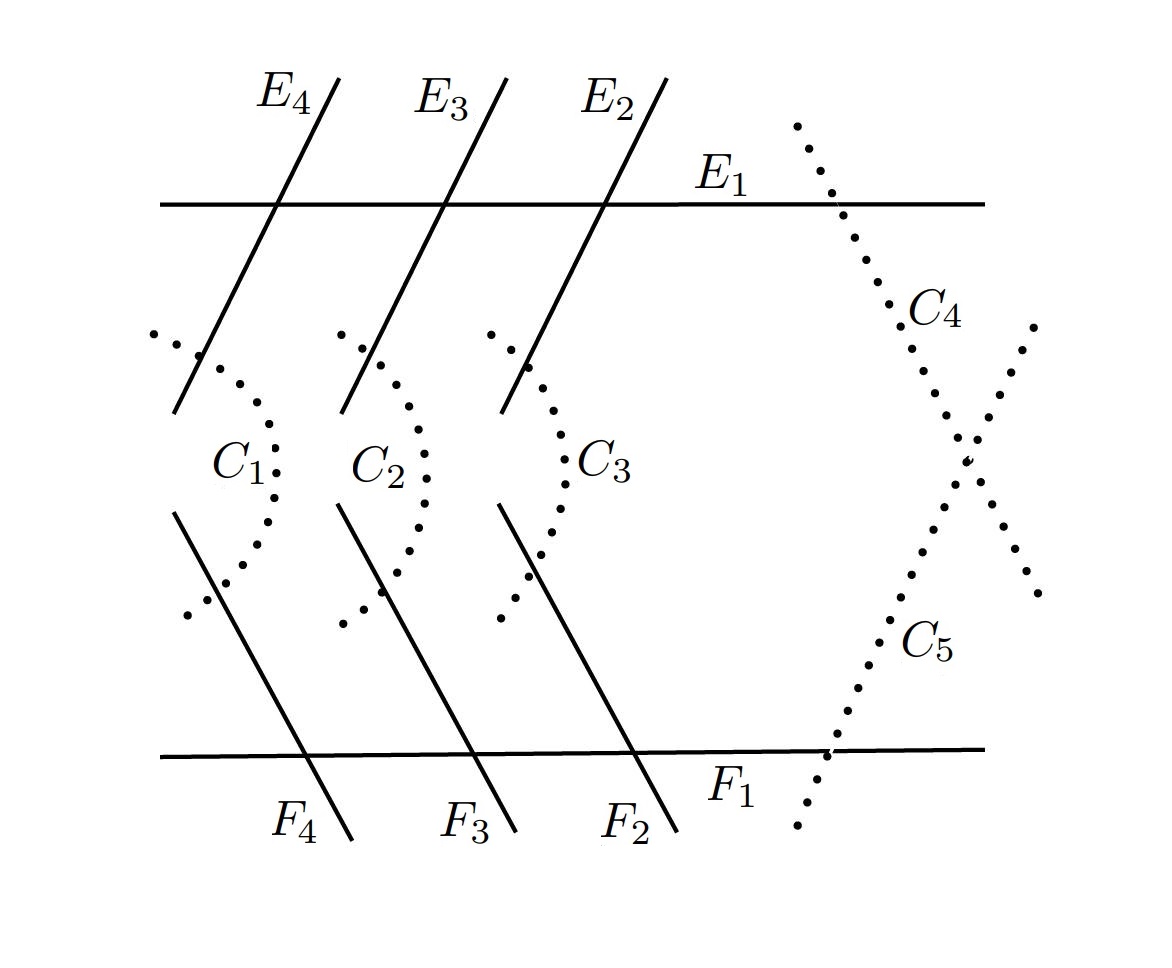}
\caption{$D_4$ + $D_4$ singularities.}
\label{fig2}
\end{figure} 
In Figure~\ref{fig2}, the solid lines correspond to the $(-2)$-curves extracted on the minimal resolution,
while the dotted lines are strict transforms of curves on $X$ which
become $(-1)$-curves on $V$. 
The curves $E_1,E_2,E_3,E_4$ are the exceptional divisors over $x$
and the curves $F_1,F_2,F_3,F_4$ are the exceptional divisors over $y$.
We can contract the $(-1)$-curves of $V$ 
repeatedly until we obtain the Hirzebruch surface $\Sigma_2$. 
We can contract the $(-1)$-curves in such a way that $X\rightarrow \Sigma_2$ is an isomorphism around $E_1$.
Moreover, we may assume that the curve
$F_{1,\Sigma_2}\subset \Sigma_2$, the image of $F_1$ on $\Sigma_2$, is a section with self-intersection two.
The image of the exceptional divisor of $X\rightarrow \Sigma_2$ on $\Sigma_2$ is four points along $F_{1,\Sigma_2}$.
We denote by $E_{i,\Sigma_2}$ the image of $E_i$ on $\Sigma_2$.
We denote by $C_{4,\Sigma_2}$ the image of $C_4$ on $\Sigma_2$.
We let $\pi\colon \Sigma_2\rightarrow \pp^1$ be the projection of $\Sigma_2$ to $\pp^1$.

The pair
$(X,C_1)$ is a reduced $2$-complement.
Locally around both $x$ and $y$, 
this is the standard reduced $2$-complement
for the $D$-type singularity.
The dual complex
$\mathcal{D}(X,C_1)$ is an interval with 
five vertices corresponding to the curves
$E_1,E_4,C_1,F_4$, and $F_1$.

We study the $1$-complements of $X$.
Let $(X,C)$ be a $1$-complement. 
By Proposition~\ref{prop:comp-klt-surf}, the curve $C$ does not contain $x$ nor $y$ on its support. 
Let $C_V$ be the strict transform of $C$ on $V$.
Then $(V,C_V)$ is a $1$-complement
and the curve $C_V$ does not intersect any of the curves $E_i$'s nor $F_i$'s.
Let $C_{\Sigma_2}$ be the push-forward of $C_V$ on $\Sigma_2$.
The following conditions are satisfied:
\begin{enumerate}
    \item the pair
    $(\Sigma_2,C_{\Sigma_2})$ is a $1$-complement, 
    \item the curve $C_{\Sigma_2}$ does not contain $E_{1,\Sigma_2}$ nor $F_{1,\Sigma_2}$ on its support, 
    \item the curve $C_{\Sigma_2}$ contains no vertical components over $\pp^1$, 
    \item every component of $C_{\Sigma_2}$ intersects $F_{1,\Sigma_2}$ in either three or four points.
\end{enumerate}
The third statement follows from the fact that $C_V$ does not intersect any of the $E_i$'s.
The last statement holds as every component of $C_{\Sigma_2}$ is horizontal over $\pp^1$ and must intersect $F_{1,\Sigma_2}$ along 
$F_{1,\Sigma_2}\cap E_{i,\Sigma_2}$
for $i\in \{2,3,4\}$.
Furthermore, every such component
may intersect $F_{1,\Sigma_2}\cap C_{4,\Sigma_2}$.

We turn to analyze the class of the components of $C_{\Sigma_2}$ in the Picard group of $\Sigma_2$.
Let $s$ be the class of the $(-2)$-section
and $f$ be the class of a fiber.
The classes $s$ and $f$ generate the cone of effective curves
and $-K_{\Sigma_2} \sim 2s+4f$.
Note that every irreducible curve 
equivalent to either $s$ or $s+f$ 
contains $E_{1,\Sigma_2}$. 
Every irreducible curve equivalent to
$s+2f$ intersects $F_{1,\Sigma_2}$ twice. 
We conclude that $C_{\Sigma_2}$ must have a single smooth component equivalent to $2s+4f$ which ramifies at the points
$F_{1,\Sigma_2}\cap E_{i,\Sigma_2}$ for
$i\in \{2,3,4\}$. 
Hence, the only possible $1$-complement $(X,C)$ of $X$ consists of a single elliptic curve $C$ which does not contain $x$ nor $y$ on its support. 
From the construction, it follows that there is a unique such a curve.
We conclude that every $1$-complement
of $X$ has coregularity one.
}
\end{example}

To conclude this subsection, 
we show an example
of a singular del Pezzo surface 
of coregularity zero
and
Picard rank one
which admits a 
$2$-complement of coregularity zero
and no $1$-complement. 
Indeed, the following lemma allows us to construct several such examples.

\begin{lemma}\label{lem:no-1-comp}
Let $X$ be a Fano type surface.
Assume that $X$ contains a D-type
singularity which is not Gorenstein.
Then, $X$ admits no $1$-complement.
\end{lemma}

\begin{proof}
Assume that $X$ is a Fano type surface
and $x\in X$ is  $D$-type singularity which is not Gorenstein.
Let $E$ be the exceptional divisor over $X$
with center $x\in X$ 
which computes the minimal log discrepancy of $X$ at $x$. 
Since $(X;x)$ is a $D$-type singularity which is not Gorenstien, 
we have that $0<a_E(X;x)<1$.
Let $(X,B)$ be a $1$-complement. 
Since $K_X+B\sim 0$, then every log discrepancy of $(X,B)$ is integral.
In particular, we must have that
$a_E(X,B;x)=0$. 
Hence, $(X,B;x)$ is a non-trivial $1$-complement of the $D$-type singularity $(X;x)$.
This contradicts Proposition~\ref{prop:comp-klt-surf}.
We conclude that $X$ admits no $1$-complement.
\end{proof}

The following is a modification of  Example~\ref{ex:d4d4}.

\begin{example}
\label{ex:d34d4}
{\em 
Let $X$ be a Gorenstein del Pezzo surface with two singularities of type $D_4$.
Let $V\rightarrow X$ be the minimal resolution of $X$.
We use the notation of Figure~\ref{fig2}.
Let $V'$ be the variety obtained by blowing-up the point 
$C_3\cap E_2$ in $V$.
We call $E$ the exceptional divisor of $V'\rightarrow V$.
We 
identify the curves on $V$
with their strict transforms on $V'$.
We obtain the diagram in Figure~\ref{fig3}.
The double curve $E_2$ is a $(-3)$-curve.
\begin{figure} 
\includegraphics[scale=0.2]{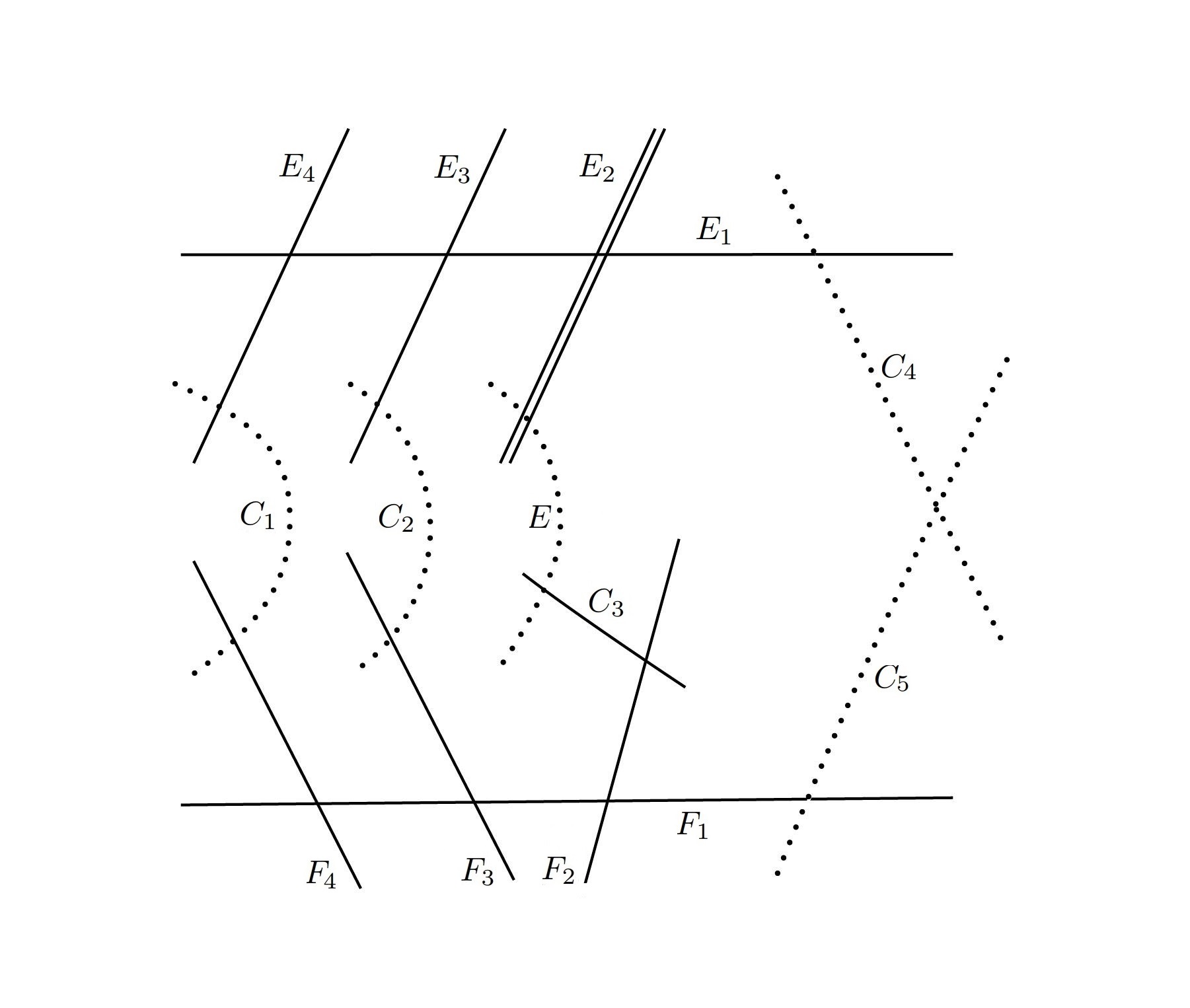}
\caption{$D$-type + $D_4$ singularities.}
\label{fig3}
\end{figure}
Let $p\colon V'\rightarrow X'$ be the 
morphism obtained by contracting the disjoint set of curves 
$\{E_1,\dots,E_4\}$ and 
$\{F_1,\dots,F_4,C_3\}$. 
Note that $C_3$ is a $(-2)$-curve on $V'$.
By abuse of notation, we use the same notation for the divisor of $V'$ which are not contracted on $X'$ and 
their images on $X'$.
Then, the image of
$F_1\cup \dots \cup F_4\cup C_3$ on $X'$ is a $D_5$-singularity $x \in X'$
and the image of 
$E_1\cup \dots \cup E_4$ on $X'$ 
is a $D$-type singularity $y\in X'$.
Since $E_2$ is a $(-3)$-curve, then
$y\in X'$ is a $D$-type singularity which is not Gorenstein.
We conclude that 
$X'$ is a singular del Pezzo surface
of Picard rank one with two singularities;
a $D$-type singularity which is not Gorenstein
and a $D_5$-singularity.
By Lemma~\ref{lem:no-1-comp}, 
we conclude that $X'$ admits not $1$-complement.
The pair $(X',E)$ is a reduced $2$-complement.
Near both $x$ and $y$, the curve
$E$ is just a reduced $2$-complement of the $D$-type singularities. 
Indeed, we can compute the pull-back: 
\[
p^*(K_{X'}+E)=
K_{V'}+E_1+E_2+E+C_3+F_2+F_1+
\frac{1}{2}(E_3+E_4+F_3+F_4).
\] 
We conclude that the dual complex
$\mathcal{D}(X',E)$
correspond to an interval with 
six vertices corresponding to the 
divisors
$E_1,E_2,E,C_3,F_2,$ and $F_1$.
Thus, $X'$ is a singular Del Pezzo surface of Picard rank one
and coregularity zero
which admits a 
$2$-complement
but no $1$-complement.
}
\end{example}

The construction of Example~\ref{ex:d34d4}
can be generalized to obtain countably
many singular del Pezzo surfaces
of Picard rank one
and coregularity zero
which admits a $2$-complement but no $1$-complement.
Indeed, we can blow up the intersection
$E\cap E_2$ or $E\cap C_3$
and proceed to blowing-up inductively the intersection of the new exceptional divisors
with the strict transforms of the previous divisors in the diagram.
Afterward, we may contract every single curve in this smooth model,
except for the last exceptional curve $E$, 
$C_1,C_2,C_4$ and $C_5$.
By doing so, we obtain a singular del Pezzo $X$ of Picard rank one 
with two $D$-type singularities. 
Furthermore, one of them is non-Gorenstein.
By Lemma~\ref{lem:no-1-comp}, the variety $X$ admits no $1$-complement.
The pair $(X,E)$ is a reduced $2$-complement. 

\subsection{Terminal 3-fold singularities}
\label{subsec:ter-3-fold}
In dimension two, terminal singularities are just smooth points.
Then, the minimal model program for a smooth surface
yields a smooth outcome.
However, the ample model (if it exists) 
of a minimal smooth surface
has canonical singularities.
In dimension two, the canonical singularities
are the Du Val singularities explained above.
In dimension three, terminal singularities
were classified by Mori (see, e.g.,~\cite[Theorem 6.1]{Rei85}).
All these singularities are hyperquotient singularities, i.e., 
the finite quotient of a hypersurface singularity.
We recall the classification of terminal $3$-fold singularities.
In the following theorem, $\mu_r$ is the group acting 
on the hypersurface. 

\begin{theorem}
Let $(X;x)$ be a terminal $3$-fold singularity.
Then, $(X;x)$ is a hyperquotient singularity
so we can write $(H;0)/\mu_r \simeq (X;x)$, 
where $H=\{f=0\}$ in $\kk[x,y,z,w]$.
Then, up to a $\mu_r$-equivariant change of variables, 
$f$ is one of the following: 
\begin{enumerate}
    \item $cA_r$ singularities: $f=xy+g(z,t)$ with $g\in m^2$ 
    or $f=x^2+y^2+g(z,t)$ with $g\in m^3$.
    \item $cD_4$ singularities: $f=x^2+g(y,z,t)$ with $g\in m^3$
    and $g_3$, the cubic part of $g$, is an irreducible cubic.
    \item $cD_n$ singularities: $f=x^2+yz^2+g(z,t)$ with $g\in m^4$.
    \item $cE$ singularities: $f=x^2+y^3+yg(z,t)+h(z,t)$ where
    $g\in m^3$ and $h\in m^4$.
\end{enumerate}
\end{theorem}

For the weights of the actions in each case, we refer the reader to~\cite{Rei85}. 
On the other hand, canonical singularities are not classified.
However, we have the following theorem due to Reid
about the general hyperplane sections
on a rational Gorenstein singularity.

\begin{theorem}
Let $(X;x)$ be a rational Gorenstein singularity.
Let $H\subset X$ be a general hyperplane section containing $x$.
Then, the singularity $(H;x)$ is a rational or elliptic 
Gorenstein singularity.
\end{theorem}

Recall that canonical Gorenstein singularities
are the same as rational Gorenstein singularities.
In particular, we obtain the following theorem for canonical $3$-fold singularities.

\begin{theorem}\label{thm:general-section} 
Let $(X;x)$ be a canonical Gorenstein singularity of dimension $3$.
Let $H\subset X$ be a general hyperplane section containing $x$.
Then, the singularity $(H;x)$ is either a Du Val singularity
or a Gorenstein elliptic singularity.
\end{theorem}

A singularity $(X;x)$ for which a general hyperplane section
$(H;x)$ is a Gorenstein elliptic singularity 
will be called a {\em compound elliptic singularity}, 
following the similar notation for compound Du Val singularities.
The following theorem can be proved using the previous classification,
Theorem~\ref{thm:general-section}, 
and a similar argument to that used in the proof of Theorem~\ref{thm:surf-coreg-0}.

\begin{theorem}
Let $(X;x)$ be a $3$-fold singularity of coregularity zero.
Then, there exists a $2$-complement $(X,B;x)$ and a finite cover
$\phi\colon Y\rightarrow X$ such that $Y$ 
is either $cA_r$ singularity,
$cD_r$ singularity, 
or a compound elliptic singularity.
\end{theorem}

As discussed in the previous section,
the quotient of a singularity may have lower coregularity.
The singularities of type
$cA_r$, $cD_r$ and compound elliptic singularities have coregularity zero.
The fact that the quotient still has coregularity zero
imposes a non-trivial condition on the acting group. 
Especially, if there is a complement
that this group is fixing, as this complements
behaves as a semi-invariant of low degree.
Let us note that in the case of dimension $3$ 
no canonical singularity is exceptional.
This is different from what happens to surfaces
in which there are already exceptional Du Val singularities.

\begin{problem}
Find the exceptional $3$-fold singularity with largest log discrepancy.
\end{problem}

It is expected that the previous example is computed by a $3$-fold quotient singularity.

\subsection{Klt $3$-fold singularities}
\label{subsec:klt-3-fold}
In this subsection, we discuss klt $3$-fold singularities
from a more general perspective via complements
and coregularity.
The following theorem follows from the theory of complements for surfaces.

\begin{theorem}\label{thm:3-fold-comp}
Let $(X,\Delta;x)$ be a $3$-fold klt singularity. 
Assume that $\Delta$ has standard coefficients.
Then, one of the following statements hold:
\begin{enumerate}
    \item the singularity $(X,\Delta;x)$ is exceptional, 
    \item the singularity $(X,\Delta;x)$ has coregularity one 
    and it admits a $1,2,3,4$ or $6$-complement, or 
    \item the singularity $(X,\Delta;x)$ has coregularity zero and it either admits
    a $1$-complement or $2$-complement.
\end{enumerate}
\end{theorem}

In the first case, we expect that
$(X,\Delta;x)$ admits a $N$-complement,
where $N\leq 66$.
In the case of exceptional singularities, 
we know that these can be put together
in bounded up to deformation families.
In the case of coregularity one,
even if the minimal log discrepancy 
is bounded away from zero, we do not expect this behavior.
For instance, one can consider
$cA_n$-singularities.
Instead, for $3$-fold singularities
of coregularity one we can attempt to 
construct certain birational models
in which the singularity is simplified.
In order to do so, 
we need to introduce the concept
of quotient-dlt model: 

\begin{definition}
{\em 
Let $(X,\Delta)$ be a log canonical pair. 
We say that $(X,\Delta)$ has {\em quotient divisorially log terminal singularities}
(or {\em qdlt} for short) if 
there exists an open subset $U\subset X$
satisfying the following:
\begin{enumerate}
    \item every log canonical center
    of $(X,\Delta)$ intersects $U$, 
    \item every log canonical center of
    $(X,\Delta)$ is a strata of $\lfloor \Delta\rfloor$, and 
    \item for every strata $Z\subset \lfloor \Delta \rfloor$ of codimension $z$ there is an isomorphism
    \[
    (X_Z,\lfloor \Delta\rfloor_Z) \simeq 
    (\mathbb{A}^z/\zz_k, H_1 + \dots + H_z),
    \]
\end{enumerate}
where $\zz_k$ is acting on $\mathbb{A}^z$ as multiplication by roots of unity
and each $H_i$ is the image of the corresponding hyperplane on $\mathbb{A}^z$.
We ask the previous isomorphism to hold at the level of formal completions, i.e., 
the subscript $Z$ of $X_Z$ denotes the corresponding formal completion
of the localization of $X$ at $Z$.
In other words, the log canonical centers of $(X,\Delta)$ may not be snc but at the very least it behaves as a finite quotient
of a snc pair.
}
\end{definition}

In the case of $3$-fold singularities
of coregularity one, we can prove the following theorem 
regarding bounded qdlt models (see, e.g.,~\cite{Mor21b}).

\begin{theorem}
Let $(X,\Delta;x)$ be a $3$-fold klt singularity
of coregularity one. 
There exists a $6$-complement 
$(X,B;x)$ which is strictly log canonical at $x$.
There exists a projective birational morphism 
$\pi\colon Y\rightarrow X$ satisfying the following:

and satisfies the following condition:
\begin{itemize}
    \item the projective morphism $\pi$ contracts two irreducible divisors $E_1$ and $E_2$, 
    \item we have that $K_Y+B_Y + E_1 + E_2 = \pi^*(K_X+B)$,
    where $B_Y$ is the strict transform of $B$ on $Y$, and
    \item the pair $(Y,B_Y+E_1+E_2)$ is qdlt.
\end{itemize}
\end{theorem}

The previous theorem can be used to compute 
invariant of singularities of $(X,\Delta)$.
The pair $(Y,B_Y+E_1+E_2)$ is refered as a {\em qdlt modification} of $(X,B)$, following the corresponding notation for dlt pairs.
In general, it is not known, even in dimension $3$
if we can find a ``bounded" qdlt model for bounded complements.
We propose this as a question that would enhance
our understanding of $3$-fold klt singularities.

\begin{problem}
Find a constant $N_3$ satisfying the following property: 
for every klt $3$-fold singularity $(X;x)$ there is a strictly
log canonical $N_3$-complement $(X,B;x)$ which admits 
a qdlt model $\pi\colon Y\rightarrow X$ for which 
$\rho(Y/X)$ is bounded above by $N_3$.
\end{problem}

\subsection{Varieties with torus actions}
\label{subsec:torus-actions}
In this subsection, we discuss the coregularity
of Fano $\mathbb{T}$-varieties.

A {\em $\mathbb{T}$-variety} is a 
normal variety $X$
endowed with the effective action of an algrebaic torus 
$\mathbb{G}_m^k$. 
The {\em torus complexity} of the $\mathbb{T}$-variety
is the dimension of the variety
minus the dimension of the acting torus.
In the case that the torus complexity 
is zero, then we say that $X$ 
is a {\em toric variety}. 
Varieties with torus action often appear in 
algebraic geometry as
the central fiber of some special degenerations. 
The torus complexity
is also the dimension
of the normalized Chow quotient
of $X$ by the acting torus.

Affine toric varieties can be described
by polyhedral cones in a rational vector space~\cite{CLS11}.
The combinatorics of this polyhedral cone determine
the geometry of the affine toric variety.
To go from the theory of affine toric varieties
to the general setting of toric varieties, 
one needs to find a way to glue these polyhedral cones.
By Sumihiro's theorem, we know that every $\mathbb{T}$-variety
is covered by affine $\mathbb{T}$-invariant open sets~\cite{Sum74}.
Using the previous theorem
and the cone description 
of affine toric varieties
leads naturally to the concept of fans 
of polyhedral cones in rational vector spaces.
Again, the combinatorics
of the fan encrypt the geometry of the toric variety. 
A projective toric variety $T$
is the prototype of a variety with coregularity zero.
Indeed, if $B_T$ is the reduced sum of the prime
torus invariant divisors, then $(T,B_T)$ is log Calabi--Yau
and the dual complex $\mathcal{D}(T,B_T)$ is a sphere 
of dimension $\dim T -1$.
Similarly, every toric singularity
has coregularity zero.

In~\cite{AH06}, the authors started a theory that generalized
the language of toric varieties
to $\mathbb{T}$-varieties in general.
In order to describe an affine $\mathbb{T}$-variety $X$,
the authors define a {\em polyhedral divisor} $\mathcal{D}$
on the normalized Chow quotient $Y$.
This is a divisor whose coefficients are rational polyhedra
instead of rational numbers. 
Then, the authors prove that there is a divisorial sheaf 
$\mathcal{A}(\mathcal{D})$ associated with this polyhedral divisor.
The spectrum of the ring of sections of $\mathcal{A}(\mathcal{D})$
recovers the variety $X$ equivariantly.
Furthermore, the authors prove that every affine $\mathbb{T}$-variety
comes from a polyhedral divisor.
In~\cite{AHS08}, the authors generalize the theory
of polyhedral divisors to the concept of divisorial fans
to recover all possible $\mathbb{T}$-varieties. 
The philosophy, is that, whenever we can understand the 
normalized Chow quotient of $X$
and the divisorial fan of $X$, we can reduce problems on $X$
to this lower-dimensional variety.
In this direction, we prove that the torus complexity 
behaves well with respect to the coregularity.

\begin{theorem}\label{thm:equiv-comp}
Let $X$ be a Fano variety of torus complexity $c$. 
Then, we have that ${\rm coreg}(X)\leq c$.
\end{theorem}

\begin{proof}
Let $\mathbb{G}_m^k$ be the acting torus
and $n$ be the dimension of $X$ so 
$c=n-k$. 
For each $r\in \zz_{\geq 2}$, we consider
$\mu_r$ the group of roots of unity of $\mathbb{G}_m^k$.
Quotients of Fano varieties by finite groups
are Fano type (see, e.g.,~\cite{Mor21}).
Hence, we conclude that $Y_r:=X/\mu_r$
is a Fano type variety.
We write $\pi_r\colon X\rightarrow Y_r$ for the quotient.
By Theorem~\ref{thm:boundedness-complements}, there is a $N_n$-complement
$(Y_r,B_r)$ which only depends on $n$, i.e., we have that
$N_n(K_{Y_r}+B_r)\sim 0$. 
We write
\[
K_X+\Gamma_r = \pi_r^*(K_{Y_r}+B_r).
\]
Then, $(X,\Gamma_r)$ is a $\mu_r$-equivariant $N_n$-complement.
This means that $N_n(K_X+\Gamma_r)\sim 0$ 
and $\mu_r\leqslant {\rm Aut}(X,\Gamma_r)$.
Since $X$ is Fano, the automorphism group ${\rm Aut}(X,\Gamma_r)$
is linear algebraic.
Since the set $\{ (X,\Gamma_r)\}_{r\in \zz_{\geq 2}}$
belongs to a log bounded family, 
then there are only finitely many possible isomorphism types for
$\{ {\rm Aut}(X,\Gamma_r)\}_{r\in \zz_{\geq 2}}$.
Hence, we can set $\Gamma_X:=\Gamma_r$ for some $r$ large enough
and assume that $\mu_r \leqslant {\rm Aut}(X,\Gamma)$ for every $r$.
This implies that $\mathbb{G}_m^k \leqslant {\rm Aut}(X,\Gamma)$. Then, we can find a $\mathbb{G}_m^k$-equivariant log resolution $\widetilde{X}$ 
of $(X,\Gamma)$ that dominates the normalized Chow quotient $Y$ of $X$ by $\mathbb{G}_m^k$.
We let $(\widetilde{X},\widetilde{\Gamma})$ be the log pull-back of $(X,\Gamma)$ to $\widetilde{X}$. 
Then, the restriction of $(\widetilde{X},\widetilde{\Gamma})$
to the general fiber $(F,\Gamma_F)$ of $\widetilde{X}\rightarrow Y$ is a $(n-c)$-dimensional toric log pair which is log Calabi-Yau. We conclude that all the components of $\Gamma_F$ appear with coefficient one. 
Hence, we conclude that on $\widetilde{X}$ the pair
$(X,\Gamma)$ has at least $n-c$ log canonical places
whose intersection dominates $Y$. 
This finishes the proof.
\end{proof}

The previous theorem implies that 
a complexity one Fano variety 
has coregularity zero or one.
In particular, in any case, it admits a $6$-complement.
Although toric singularities are always klt type,
singularities of torus complexity one 
are not necessarily log terminal, rational, Cohen-Macaulay, nor Du Bois (see, e.g.,~\cite{LS13,LLM18,LLM19,LLM20}).
The coregularity and torus complexity have a
common goal: reduce higher-dimensional geometry problems
to low-dimensional geometry and combinatorics.

\begin{theorem}
Let $X$ be a Fano type variety of torus complexity one
and coregularity zero.
Then $X$ admits an equivariant $2$-complement $(X,B)$
and a finite cover $Y\rightarrow X$ of degree
at most two for which the log pull-back $(Y,B_Y)$ of $(X,B)$
degenerates to a projective toric variety.
\end{theorem}

\begin{proof}
The normalized Chow quotient of $X$ is $\pp^1$.
We can find an equivariant projective birational map
$\widetilde{X}\rightarrow X$ for which
the torus action admits a good quotient to $\pp^1$.
We denote the good quotient by $q\colon \widetilde{X}\rightarrow \pp^1$.
Let $\widetilde{\Gamma}$ be the sum of all the 
torus invariant divisors that dominate $\pp^1$.
Then the log pair
$(\widetilde{X},\widetilde{\Gamma})$
has a log canonical center which dominates $\pp^1$
and is isomorphic to $\pp^1$.
We denote one of such log canonical centers $C$.
Since $X$ has coregularity zero, 
the pair defined by adjunction
$(K_{\widetilde{X}}+\widetilde{\Gamma})|_C = K_C + \Gamma_C$
has coregularity zero as well.
Hence, we conclude that
$\Gamma_C$ has either two
fractional coefficients
or three fractional coefficients
of the form $1-\frac{1}{p},\frac{1}{2}$, and $\frac{1}{2}$.
Note that the pair $(\widetilde{X},\widetilde{\Gamma})$
is log Calabi-Yau over $\pp^1$.
The log pair obtained by the canonical bundle formula
is isomorphic to $(C,\Gamma_C)$.
We can find a $2$-complement for $(C,\Gamma_C)$
which we denote $(C,\Gamma_C+B_C)$.
Then, the pair $(\widetilde{X},\widetilde{\Gamma}+q^*B_C)$
is reduced $2$-complement.
We define $B$ to be the push-forward of $\widetilde{\Gamma}+q^*B_C$ to $X$.
Then, $(X,B)$ is a $2$-complement.
We take the index one cover of $K_X+B$ and call it $Y$.
Then, the log pull-back $(Y,B_Y)$ is a $1$-complement
and $Y$ is a Fano type variety of torus complexity one.
Then, any torus invariant affine variety of $Y$ 
is described by a polyhedral divisor
on $\pp^1$ with exactly two fractional coefficients.
By~\cite[Theorem 2.8]{IV12}, each such invariant affine variety
can be deformed into an affine toric variety.
These deformations are compatible, so they glue together
to a deformation of $(Y,B_Y)$ to a  projective toric variety.
\end{proof}

\subsection{Coregularity under morphisms}
\label{subsec:coreg-morphisms}
In subsection~\ref{subsec:quotients}, 
we studied the behavior
of the coregularity under finite quotients.
In this subsection, we study its behavior under birational
and contractions.
We start with the first proposition, which states that the coregularity
can only drop under a birational contraction.

\begin{proposition}\label{prop:coreg-bir}
Let $X$ be a variety
of coregularity $c$.
Let $X\dashrightarrow Y$ be a birational contraction.
Then, the coregularity of $Y$ is at most $c$.
\end{proposition}

\begin{proof}
If the coregularity of $c$ is infinite, i.e., $X$ does not admit a log Calabi--Yau structure, 
then the proposition holds trivially.
Assume that the coregularity of $X$
is $c\in \zz_{>0}$.
Let $\Delta$ be a boundary on $X$
which computes the coregularity. 
We denote by $\Delta_Y$ the push-forward of $\Delta$ on $X$.
Note that $(X,\Delta)$ is log Calabi-Yau.
Let $p\colon Z\rightarrow X$
and $q\colon Z\rightarrow Y$
be a common log resolution
of both $(X,\Delta)$ and $(Y,\Delta_Y)$.
Then, we can write 
\[
p^*(K_X+\Delta)-q^*(K_Y+\Delta_Y)=E-F,
\]
where $E$ and $F$ are two
$q$-exceptional effective divisors
with no common components. 
Applying the negativity lemma twice,
we conclude that $E=F=0$.
Hence, we have that
$(X,\Delta)$
and 
$(Y,\Delta_Y)$ are log crepant equivalent.
In particular,
$(Y,\Delta_Y)$ is a log Calabi--Yau pair
and 
${\rm coreg}(X,\Delta)={\rm coreg}(Y,\Delta_Y)$. 
This finishes the proof.
\end{proof}

Observe that the inequality
between the coregularities proved 
in Proposition~\ref{prop:coreg-bir} can be strict.
Indeed, a del Pezzo of degree one $X_1$
admits a projective birational morphism
to $\pp^2$. 
The coregularity of $\pp^2$ is zero
while
the coregularity of $X_1$ is one.
The following proposition 
explains the behavior of the coregularity
under fibrations, i.e., 
morphisms between normal
projective varieties with connected fibers.
In the following proof, we will use the language of generalized pairs.

\begin{proposition}
Let $\phi\colon X\rightarrow Y$ be a fibration. 
Assume that $X$ is of Fano type.
Then, we have that
the inequality
\[
{\rm coreg}(X) \geq {\rm coreg}(Y) 
\]
holds.
In particular, if $X$ has coregularity zero, then $Y$ has coregularity zero.
\end{proposition}

\begin{proof}
Let $\Delta$ be a boundary on $X$ which computes the coregularity.
Then, $(X,\Delta)$ is a log Calabi--Yau pair.
In particular, we have that
$K_X+\Delta \sim_{Y,\qq} 0$.
Let $(Y,\Delta_Y+M_Y)$ be the pair obtained by
the canonical bundle formula. 
Then, $(Y,\Delta_Y+M_Y)$ is a generalized log canonical pair.
By construction, we have that 
\[
K_X+\Delta \sim_{\qq} \phi^*(K_Y+\Delta_Y+M_Y).
\] 
In particular, $(Y,\Delta_Y+M_Y)$
is a generalized log Calabi--Yau pair.
By~\cite[Theorem 2.9]{FS20}, we can take a generalized dlt modification
$(Y',\Delta_{Y'}+M_{Y'})$ of
$(Y,\Delta_Y+M_Y)$.
Proceeding as in~\cite[Lemma 2.36]{Mor21}, we obtain a commutative diagram,
\[
\xymatrix{
(X,\Delta)\ar[d]_-{\phi} & (X',\Delta')\ar@{-->}[l]\ar[d] \\
(Y,\Delta_Y+M_Y) &
(Y',\Delta_{Y'}+M_{Y'})\ar[l]
}
\] 
where $(X',\Delta')$ is a dlt modification of $(X,\Delta)$.
By~\cite{Fil18}, we know that every log
canonical cneters of $(X',\Delta')$
maps to a generalized log canonical center
of $(Y',\Delta_{Y'}+M_{Y'})$.
Since $(X,\Delta)$ has coregularity $c$, then $(X',\Delta')$ has a log canonical center of dimension $c$.
In particular, 
we conclude that 
$(Y',\Delta_{Y'}+M_{Y'})$ has a generalized log canonical center of dimension at most $c$.
In particular, the minimal dlt center
of $(Y',\Delta_{Y'}+M_{Y'})$ has dimension at most $c$.
This implies that the generalized pair
$(Y,\Delta_Y+M_Y)$ has coregularity at most $c$.

Now, we turn to prove that 
$Y$ has coregularity at most $c$. 
In order to do so, we want to turn
the nef part $M_{Y'}$ into an effective divisor.
By~\cite[Lemma 2.12]{Bir19}, we have that $Y$ is of Fano type.
We conclude that $Y'$ is also of Fano type as well.
In particular, $Y'$ is a Mori dream space.
The diminished base locus of $M_{Y'}$ has codimension at least two. 
We run a $M_{Y'}$-MMP which terminates with a good minimal model $Y'\dashrightarrow Y''$. 
All the steps of this minimal model are flips. 
Furthermore, we have that 
$M_{Y''}$ is a semiample divisor.
Let $\Delta_{Y''}$ be the push-forward
of $\Delta_{Y'}$ to $Y''$.
Then, $(Y'',\Delta_{Y''}+M_{Y''})$ is a generalized log Calabi--Yau pair,
with generalized dlt singularities,
and $M_{Y''}$ is semiample.
Hence, for a general effective element
$0\leq \Gamma_{Y''}\sim_\qq M_{Y''}$
the pair
$(Y'',\Delta_{Y''}+\Gamma_{Y''})$
has dlt singularities.
Since $\lfloor \Delta_{Y''}\rfloor$
has at least $c$ components with a common intersection point, so
the same holds for 
$\lfloor \Delta_{Y''}+\Gamma_{Y'}\rfloor$.
We conclude that 
$(Y'',\Delta_{Y''}+\Gamma_{Y''})$ is a dlt pair of coregularity at most $c$.
Let $\Gamma_{Y'}$ be the strict transform of $\Gamma_{Y''}$ on $\Gamma_{Y'}$.
Then, the pair
$(Y',\Delta_{Y'}+\Gamma_{Y'})$
is a log canonical pair of coregularity at most $c$.
Thus, we have that ${\rm coreg}(Y')\leq c$.
By Proposition~\ref{prop:coreg-bir}, we conclude that ${\rm coreg}(Y)\leq c$.
This finishes the proof.
\end{proof}

We note that the previous proof holds 
without the assumption that $X$ is of Fano type, provided that we have a canonical bundle formula for log canonical pairs (see, e.g.,~\cite{Amb05}). 
We also mention that the previous inequality can be strict.
Indeed, the del Pezzo surface of degree one $X_1$ admits a fibration to $\pp^1$.
The the elliptic curve $C$ of the $1$-complement $(X_1,C)$ dominates $\pp^1$. 
In the following example, we show that in general, the coregularity of the general fiber of a fibration may be higher than the coregularity of the
domain. 

\begin{example}
{\em 
Let $X=\pp^1 \times \pp^1$. 
Consider the projection
$p\colon X\rightarrow \pp^1$ onto the first component.
We define the divisors
\begin{align*}
    \Delta_{\rm vert} & = 
    (\{0\}\times \pp^1) +
    \frac{1}{2}(\{\infty\}\times \pp^1), \text{ and }
    \\
    \Delta_{\rm hor} & =
    \frac{1}{2}(\pp^1\times \{0\}) +
    \frac{1}{2}(\pp^1\times \{1\}) +
    \frac{1}{2}(\pp^1\times \{\infty\}) +
    \frac{1}{2}D,
\end{align*}
where $D$ is the diagonal.
The general fiber is
$\pp^1$ with four points of coefficient
$\frac{1}{2}$, so its coregularity is one.
On the other hand the coregularity
of $(X,\Delta_{\rm vert}+\Delta_{\rm hor})$ is zero.
}
\end{example}

In~\cite[Example 5.4]{FMP22}, the authors show examples in which the coregularity of the general fiber equals the dimension of the general fiber
while the pair itself has coregularity zero.

\subsection{Coregularity under deformations} 
\label{subsec:coreg-deform}

In this subsection, we discuss the behaviour
of the coregularity under deformations and degenerations. 
The following example shows that the coregularity
can decrease on the special fiber of a flat family
of singularities.

\begin{example}
{\em 
Consider the canonical
exceptional Brieskorn singularity 
\[
x^3+y^3+z^4+w^5=0.
\] 
Then, consider the following smoothing:
\[
\mathcal{X}:=
\{
x^3+y^3+z^4+w^5+tw=0 
\} \rightarrow \mathbb{A}^1_t.
\] 
The central fiber $\mathcal{X}_0$ is exceptional
so its coregularity equals two. 
On the other hand, for every $t\neq 0$, we have that
$\mathcal{X}_t$ is smooth at the origin, so 
its coregularity is zero. 
}
\end{example}

Similar examples show that coregularity zero
$n$-dimensional singularities can 
degenerate to exceptional $n$-dimensional singularities, i.e., 
singularities of dimension $n$
and coregularity $n-1$.
The following example shows that the coregularity
can increase in the special fiber of a flat family
of Fano varieties.

\begin{example}
{\em
Let $X_d$ be a del Pezzo surface of degree $d$. 
Fix $d\leq 4$.
We may find a flat family 
$\mathcal{X}\rightarrow \mathbb{A}^1$
for which $\mathcal{X}_0$ is a toric varitey
and $\mathcal{X}_1\simeq X$.
For instance, we can consider the blow-up
$X\rightarrow \pp^2$
and deform the blown-up points 
in such a way that 
each blow-up occurs at a torus invariant point of
$\pp^2$.
Hence, for  there is a flat family $\mathcal{X}$
for which the central fiber $\mathcal{X}_0$ 
is a Fano type variety
of coregularity zero
and nearby fiber $\mathcal{X}_t$, with $t\neq 0$, 
are del Pezzo surfaces of degree at most $4$, 
so they have positive coregularity.
}
\end{example}

\subsection{Properties of dual complexes} 
\label{subsec:prop-dual-complexes}

In this subsection, we discuss some further properties 
about dual complexes. 
In order to introduce the first theorem of this subsection, 
we need to recall the concept of standard 
$\pp^1$-link.

\begin{definition}
{\em 
Let $X\rightarrow S$ be a projective contraction.
Let $(X,D_1+D_2+\Delta)$ be a pair.
A morphism $X\rightarrow T$ over $S$ is said to be a 
{\em standard $\pp^1$-link} if the following conditions are satisfied: 
\begin{enumerate}
    \item we have that 
    $K_X+D_1+D_2+\Delta\sim_{\qq,T} 0$, 
    \item the morphism $\pi$ induce isomorphisms
    $\pi|_{D_i} \colon D_i\rightarrow T$ for each $i$,  
    \item the pair $(X,D_1+D_2+\Delta)$ is plt, and 
    \item every reduced fiber of $\pi$ is isomorphic to $\pp^1$.
\end{enumerate}
}
\end{definition}

\begin{remark}
{\em 
If $(X,\Delta)$ is a log Calabi--Yau pair 
which is birational to standard $\pp^1$-link,
then the dual complex
$\mathcal{D}(X,\Delta)$
is just two points. 
In particular, it is not connected.
}
\end{remark}

The following theorem states that the previous
is essentially the only case in which 
the dual complex of a log Calabi--Yau pair can be disconnected.
The following theorem is proved by Koll\'ar and Kovacs
in the setting of pairs~\cite{KK10}.
In the case of generalized pairs
it is proved by Filipazzi and Svaldi
and independently by Birkar~\cite{FS20,Bir20}.

\begin{theorem}\label{thm:conn}
Let $f\colon X\rightarrow S$ be a projective morphism.
Let $(X,\Delta)$ be a log Calabi--Yau pair over $S$.
Fix $s\in S$.
Assume that $f^{-1}(s)$ is connected.
Assume moreover that 
\begin{equation}\label{eq:non-klt} 
f^{-1}(s)\cap {\rm nklt}(X,\Delta) 
\end{equation}
is disconnected.
Then, the set~\eqref{eq:non-klt}
has exactly two components.
Moreover, there exists a dlt modification of $(X,\Delta)$
which is birational to a standard $\pp^1$-link over $s\in S$
up to \'etale base change.
\end{theorem}

It is easy to show that the statement of Theorem~\ref{thm:conn}
does not hold if we drop the log Calabi--Yau assumption
as shown in the following example.

\begin{example}
{\em
Given a point $p\in \pp^2$, we define the following curves:
\begin{enumerate}
\item we let $C_p:=\frac{1}{n}(C_{1,p}+\dots+C_{n,p})$ be the average
of $n$ smooth conics through $p$ with the same tangent at $p$, and 
\item we let $L_p:=\frac{1}{n}(L_{1,p}+\dots+L_{n,p})$ be the average
of $n$ lines through $p$ with different tangent direction at $p$ 
and also different from the tangent direction of the conics. 
\end{enumerate}
Now, for $k$ different points $p_1,\dots,p_k$ in $\pp^2$, 
we can consider the log pair
\[
(\pp^2, C_{p_1}+\dots+C_{p_k}+L_{p_1}+\dots+L_{p_k}).
\]
As explained in Example~\ref{ex:example-coreg-0}, the previous pair
is log canonical and its minimal dlt modification extracts
$2k$ curves so that the pre-image of each $p_i$ is the union 
of two such curves. 
Then, the dual complex 
\[
\mathcal{D}(\pp^2,C_{p_1}+\dots+C_{p_k}+L_{p_1}+\dots+L_{p_k})
\] 
is just $k$ points. 
The previous pair is log Calabi--Yau if and only if $k=1$.
}
\end{example}

Now, we turn to introduce the concept of
$\pp^1$-linking of log canonical centers.

\begin{definition}
{\em 
Let $(X,\Delta)$ be a dlt pair.
Let $Z_1$ and $Z_2$ be two log canonical centers of $(X,\Delta)$.
We say that $Z_1$ and $Z_2$ are {\em direct $\pp^1$-linked }
if there exists a log canonical center $W$ of $(X,\Delta)$
containing both $Z_1$ and $Z_2$ so that the dlt pair
$(W,\Delta_W)$ obtained from adjunction of $(X,\Delta)$ to $W$
is birational to a standard $\pp^1$-link.
Due to Theorem~\ref{thm:conn}, if $Z_1$ and $Z_2$ are direct $\pp^1$-linked,
then they have the same dimension $c$ and 
$W$ must have dimension $c+1$.
We may say that $W$ is a {\em linking center}.
Furthermore, $Z_1$ and $Z_2$ are the only log canonical centers of $(W,\Delta_W)$.

The concept of direct $\pp^1$-linking induces an equivalence relation 
on the set of log canonical centers of $(X,\Delta)$,
where we assume that every log canonical center 
is direct $\pp^1$-linked to itself.
We say that two log canonical centers of $(X,\Delta)$
are {\em $\pp^1$-linked} if they belong to the same class 
of this equivalence relation.
Note that two $\pp^1$-linked centers are birational to each other.
}
\end{definition}

In the case of a toric log Calabi-Yau pair
the minimal dlt centers are just points 
and the linking centers are just the torus invariant curves.
In a similar vein as Theorem~\ref{thm:conn}, it is proved 
that two minimal dlt centers are birational equivalent.
Even further, they are $\pp^1$-linked.

\begin{theorem}
Let $f\colon X\rightarrow S$ be a projective contraction.
Let $(X,\Delta)$ be a dlt pair
which is log Calabi-Yau over $S$.
Let $s\in S$ be a closed point
and assume that $f^{-1}(s)$ is connected.
Let $Z\subset X$ be a log canonical center of $(X,\Delta)$
which is minimal with respect to the inclusion
among centers intersecting $f^{-1}(s)$.
Let $W$ be a log canonical center of $(X,\Delta)$
for which $s\in f(W)$. 
Then, there exists a log canonical center 
$Z_W\subset W$ of $(X,\Delta)$
for which $Z$ and $Z_W$ are $\pp^1$-linked and 
$f(Z_W)$ contains $s$.
In particular, two minimal log canonical centers of $(X,\Delta)$
whose image on $S$ contain $s$ are $\pp^1$-linked.
\end{theorem}

Note that the concept of $\pp^1$-linking 
fits naturally with the definition of pseudo-manifold.

\begin{definition}
{\em 
A topological space $X$ with a triangulation $K$ is a {\em $n$-dimensional pseudo-manifold}
if the following conditions hold:
\begin{enumerate} 
\item we have that $X=|K|$, i.e., $X$ is the union of all the $n$-simplices, 
\item every $(n-1)$-simplex is the face of either one or two $n$-simplices for $n>1$, and 
\item for every pair $\sigma$ and $\sigma'$ of $n$-simplices in $K$, 
there is a sequence of $n$-simplices 
$\sigma=\sigma_0,\dots, \sigma_k=\sigma'$
such that the intersection
$\sigma_i\cap \sigma_{i+1}$
is a $(n-1)$-simplex
for every $i\in \{0,\dots,k-1\}$.
\end{enumerate}
The $n$ in the definition is called the {\em dimension} of the pseudo-manifold.
}
\end{definition}

The previous theorems together with the work of Koll\'ar and Xu gives
us the following structural theorem for dual complexes of log Calabi--Yau pairs.

\begin{theorem}
Let $(X,\Delta)$ be a log Calabi--Yau pair.
Then, the dual complex
$\mathcal{D}(X,\Delta)$ is a pseudo-manifold (possibly with boundary).
Furthermore, exactly one of the following cases hold:
\begin{enumerate}
    \item the dual complex $\mathcal{D}(X,\Delta)$
    is disconnected and it consists of two points, 
    \item the dual complex $\mathcal{D}(X,\Delta)$
    is connected and it is collapsible to a point, or
    \item the dual complex $\mathcal{D}(X,\Delta)$
    is connected, non-collapsible, and 
    \[
    H^i(\mathcal{D}(X,\Delta),\qq)=0 
    \text{ for }
    0<i<\dim \mathcal{D}(X,\Delta).
    \]
\end{enumerate}
\end{theorem}

We conclude this subsection with the following theorem due to Nakamura
that investigates the behavior of dual complexes
of numerically trivial pairs
whose singularities are worse than log canonical~\cite{Nak21}.

\begin{theorem}
Let $(X,\Delta)$ be a pair.
Assume that $(X,\Delta)$ 
has worse than log canonical singularities
and $K_X+\Delta\sim_\qq 0$.
Then, the dual complex $\mathcal{D}(X,\Delta)$ is collapsible.
\end{theorem}

In the previous theorem, we need a concept of dual complexes
for pairs whose singularities are worst than log canonical.
The definition of dual complex in subsection~\ref{subsec:DC}
extends naturally by taking a dlt modification of a possibly non-lc pair, i.e., 
the dual complex stores the combinatorial information of all
log canonical places with non-positive log discrepancy.

\subsection{Examples of dual complexes} 
\label{subsec:ex-dual-complexes}

In this subsection, we discuss 
some examples of dual complexes.
The first is the example of smooth toric pairs.

\begin{example}\label{ex:toric}
{\em 
Let $N$ be a free finitely generated abelian group
and $N_\qq:=N\otimes_\zz \qq$ be the associated $\qq$-vector space.
Let $P\subset N_\qq$ be a smooth polytope. 
We can consider the associated fan $\Sigma_P$. 
This associated fan corresponds to a projective toric variety 
$X(\Sigma_P)$. 
Let $\Delta(\Sigma_P)$ be the reduced torus invariant divisor, i.e.,
the reduced sum of all the 
prime torus invariant divisors.
Then, we have that 
\[
\mathcal{D}(X(\Sigma_P),\Delta(\Sigma_P))\simeq \partial P. 
\]
The pair $(\mathcal{D}(X(\Sigma_P),\Delta(\Sigma_P))$ is already smooth, 
so the previous equality actually holds with the given triangulation. 
Instead, if $P$ is a simplicial polytope, then 
$X(\Sigma_P)$ has $\qq$-factorial singularities
and the pair
$(X(\Sigma_P),\Delta(\Sigma_P))$
is qdlt. 
In this case, the previous equality holds up to possibly
performing some cuts on the polytope $P$.
}
\end{example}

In order to construct more interesting examples, 
one can start from a projective toric variety $X$
admitting the action of a finite group $G$
that preserves the torus
and consider the quotient $X/G$.
Since $G$ preserves the torus, then it also preserves
the torus invariant boundary. 
Thus, the quotient $X/G$ admits a complement which 
makes it into a pair of coregularity zero.
In this case, we expect the dual complex
of the induced pair structure on the quotient 
to equal $\delta P/H$ where $H$ is a homomorphic image of $G$.
In some cases, we may have $G=H$.
In the following, we show one such example.

\begin{example}
{\em 
Let $X_n:=(\pp^1)^n$ with coordinates
$([x_1:y_1],\dots,[x_n:y_n])$.
Let $\Delta_n$ be the torus invariant boundary.
By Example~\ref{ex:toric}, know that 
\[
\mathcal{D}(X_n,\Delta_n)\simeq \partial [0,1]^n.
\] 
Consider the involution
$\tau \colon X_n\rightarrow X_n$ given by
\[
\tau([x_1:y_1],\dots,[x_n:y_n])=
([y_1:x_1],\dots,[y_n:x_n]).
\]
Observe that $\tau$ preserves the torus. 
Furthermore, the only fixed points of $\tau$ 
are points of the torus.
Let $Y_n:=X_n/\tau$
and denote by $p$ the quotient morphism.
Let $\Gamma_n$ be the boundary on $Y_n$ for which 
\[
K_{X_n}+\Delta_n = p^*(K_{Y_n}+\Gamma_n).
\] 
Then, the pair $(Y_n,\Gamma_n)$ is dlt and 
\[
\mathcal{D}(Y_n,\Gamma_n)\simeq 
\mathcal{D}(X_n,\Delta_n)/\tau \simeq \mathbb{P}^{n-1}_\rr.
\]
If $n=3$, then this is an example of a reduced $2$-complement 
without boundary.
In dimension two every $2$-complement of coregularity zero
has an associated dual complex with boundary.
}
\end{example}

\section{Questions}
In this section, we present a handful of questions about the coregularity of Fano varieties.
First, we start with some structural questions regarding
the coregularity,
complements, and dual complexes.

\subsection{Structural questions}
Birkar proved that $n$-dimensional Fano type varieties admit $N_n$-complements (see, e.g.,~\cite{Bir19}). 
The strategy often reduces
to either one of the following cases:
exceptional Fano varieties,
lifting complements from the canonical bundle formula, or 
lifting complements from a non-klt center.
We expect that in most cases, we can produce a bounded complement
which comes from the minimal log canonical center of a $\qq$-complement.
In other words, if we can produce a log canonical center of dimension $c$, 
then we should be able to effectively
produce a log canonical center of dimension $c$.
This leads to the following conjecture about complements for Fano type varieties
with bounded coregularity.

\begin{conjecture}
Let $c$ be a nonnegative integer.
Let $s(c)$ be the $c$-th Sylvester's number.
Let $X$ be a Fano type variety of coregularity at most $c$.
Then, the following statements hold:
\begin{itemize} 
\item If $c=0$, then $X$ admits a $2$-complement.
\item If $c\geq 1$, then $X$ admits a $(2s(c)-3)(s(c)-1)$-complement.
\end{itemize}
\end{conjecture}

Observe that in the previous conjecture
we do not bound the dimension of $X$.
Furthermore, there is a conjectural explicit bound for the largest complement that we can find in coregularity $c$ 
in terms of Sylvester's numbers.
We mention that the previous conjecture is motivated by two recent results. 
First, in~\cite{FMP22}, Fernando Figueroa, Junyao Peng, and the author
generalized results about log canonical thresholds
which formerly depended on the dimension of the germ 
to results that only depend on the coregularity.
On the other hand, in~\cite{ETW22}, the authors give an example
of a $c$-dimensional log Calabi-Yau klt variety with index
$(2s(c)-3)(s(c)-1)$.
This result is related to previous work in which many existential bounds in algebraic geometry are proved to be either double logarithmic or double exponential~\cite{TW21,ETW21}.
Observe that in the cases of low coregularity
$0,1$ or $2$, we expect $2$-complements, $6$-complements,
and $66$-complements, respectively.
However, Sylvester's sequence grows doubly-exponentially.
Already its $7$th term is larger than $10^{12}$.
Nevertheless, in the case of Fano type varieties of coregularity zero
we expect a very nice behavior:

\begin{conjecture}\label{conj:structural-coreg-0}
Let $X$ be a $n$-dimensional Fano type variety of coregularity zero.
Then, there exists a $2$-complement $(X,B)$ and a finite cover
$Y\rightarrow X$ satisfying the following: 
\begin{itemize}
    \item The log pull-back $(Y,B_Y)$ of $(X,B)$ to $Y$ is a log Calabi-Yau pair, and 
    \item we have a PL-homeomorphism $\mathcal{D}(Y,B_Y)\simeq S^{n-1}$.
\end{itemize} 
\end{conjecture} 

Example~\ref{ex:d34d4}, shows that the $2$-complement in the statement is indeed necessary. 
On one hand, we know that all projective toric varieties
have coregularity zero. 
Further, if we consider their reduced torus invariant boundary, 
which is indeed a $1$-complement, 
we obtain a log Calabi-Yau pair whose dual complex 
is a PL-sphere.
In~\cite{Kal20}, the author gives an example of a log Calabi-Yau pair of Fano type whose dual complex is a sphere, but
it is not birational to a toric pair.
The previous conjecture says that even if coregularity zero
Fano type varieties are not necessarily toric, 
we can still find a $2$-complement 
which up to a cover behaves like a torus boundary 
of a toric variety.
Note that this $2$-complement can be described as a $1$-complement quotient by an involution. 
Indeed, a natural way to construct $2$-complemented Fano type varieties is to quotient projective toric varieties by involutions (which are not contained in the torus action).
We do not see Conjecture~\ref{conj:structural-coreg-0}
as an endgame itself, but rather a starting point to generalize
the theory of toric varieties
to coregularity zero Fano varieties.
Indeed, if we enrich the structure of $\mathcal{D}(Y,B_Y)$ with certain normal bundles of the strata, one obtains a fan-like structure $\Sigma(Y,B_Y)$ which can be used to describe much of the geometry of $Y$ (and hence $X$) itself.
In the case of coregularity one, we
expect a similar behaviour than Conjecture~\ref{conj:structural-coreg-0}.
However, in this case we expect a $6$-complement
and the dual complex to be the quotient of a sphere
of dimension $n-2$.

The previous discussion about bounded complements 
and coregularity zero is also expected for higher coregularity.
However, in such a case, the dual complex itself will retrieve
less information about the Fano variety $X$.
For instance, the dual complex will tell us nothing about
the minimal log canonical centers. 
In order to introduce the next conjecture, 
we define the concept of minimal dlt center.

\begin{definition}
{\em 
Let $(X,B)$ be a log Calabi--Yau pair. 
A {\em minimal dlt center} of $(X,B)$ is a 
minimal log canonical center of
any dlt modification of $(X,B)$.
}
\end{definition}

By the work of Koll\'ar and Kovacs~\cite{KK10}, we know that any two minimal dlt centers of dimension $c$ are $\pp^1$-linked, i.e., 
they can be connected by a sequence of subvarieties of dimension $c+1$ that are birational to $\pp^1$-bundles (see subsection~\ref{subsec:prop-dual-complexes}). 
This result implies that the birational class
of a minimal dlt center of a log Calabi-Yau pair is well-defined.
Of course, as usual, we try to study log Calabi-Yau structures
with minimized coregularity.
The following conjecture is somewhat related to the boundedness of Fano varieties.

\begin{conjecture}\label{conj:bir-bound-mlcc}
Let $c$ and $N$ be positive integers. 
There exists a birationally bounded family $\mathcal{M}_{c,N}$
of algebraic varieties satisfying the following.
Let $X$ be a Fano type variety of coregularity $c$.
Let $(X,B)$ be a $N$-complement of $X$ computing the coregularity.
Then, the minimal dlt centers of $(X,B)$ belong to 
$\mathcal{M}_{c,N}$.
\end{conjecture}

Observe that if the minimal dlt center $M$
of $(X,B)$ is rationally connected, 
then we get a 
rationally connected 
klt log Calabi-Yau pair $(M,B_M)$ of bounded index. 
By~\cite{BDCS20}, we know that these belong to birationally bounded families. 
However, in general, we may get minimal dlt centers
that are not rationally connected (see, e.g.,~\cite{San21}).
We also expect the previous conjecture to hold
if the coefficients of $B$ are controlled in a set
satisfying the DCC set (i.e., the $N$-complement hypothesis can be weakened).
In summary, Conjecture~\ref{conj:structural-coreg-0} and Conjecture~\ref{conj:bir-bound-mlcc} all-together imply that Fano type varieties of bounded coregularity
admit bounded complements
with birationally bounded minimal dlt centers.
It is worth mentioning that in the case of coregularity zero and one the previous conjecture is trivial.
However, even in dimension $2$, it is not clear we can only obtain a bounded family of K3 surfaces
as minimal dlt centers of coregularity two Fanos.

Let $(X,B)$ be a log Calabi-Yau pair of coregularity zero.
In~\cite{KX16}, Koll\'ar and Xu proved
that every strata of $(X,B)$
is a rationally connected variety.
In the case of coregularity one, 
a similar argument shows that
every two general points in each
strata can be joined
by a sequence
of rational curves
and possibly a single elliptic curve.
This motivates the following question
relating the coregularity
with the covering genus of algebraic varieties.

\begin{problem}
Let $c$ be a nonnegative integer.
Show the existence of a constant $g(c)$
satisfying the following.
Let $(X,B)$ be a log Calabi--Yau pair
of coregularity $c$.
Then, every log canonical center of $(X,B)$
has covering genus at most $g(c)$.
\end{problem}

Even though complements
of minimal coregularity seems to be the 
best choice in most cases,
there could be more than one complement
computing the coregularity.
For instance, in $\pp^2$
we can choose any three transversal lines
as we can choose any conic and a line. 
It is somewhat clear that the most 
components the complement has the better. 
For instance, in 
$\pp^n$ any two $1$-complements 
with $(n+1)$ components differ by an automorphism.
Similar principles hold for toric varieties.
In the next subsection, we introduce an invariant that allows us to make a precise statement 
about the number of components of a complement.

\subsection{Complexity of Fano varieties}
The complexity of a Fano variety, 
or more generally a Fano type morphism, 
measures the difference between 
the coefficients of the complements
and the dimension plus the Picard rank.
We give a precise definition in the following.
In order to do so, we introduce
Fano type morphisms and log Calabi--Yau morphisms.

\begin{definition}
{\em 
Let $\phi\colon X\rightarrow Z$ be a projective contraction.
We say that $\phi$ is a {\em Fano type morphism}
if there exists a boundary $\Delta$ on $X$, big over $Z$, for
which $(X,\Delta)$ is klt and $K_X+\Delta\sim_{\qq,Z}0$.
If $Z$ is a point, then a Fano type morphism
is simply a Fano type variety.
If $X\rightarrow Z$ is the identity and $z\in Z$ a closed point, then a Fano type
morphism around $z\in Z$ is simply a klt type singularity structure
around the point $z$.

We say that $\phi$ is a {\em log Calabi--Yau morphism}
if there exists a boundary $B$ on $X$
for which $(X,B)$ is log canonical
and $K_X+B\sim_{\qq,Z}0$. 
If $Z$ is a point, then a log Calabi--Yau morphism
is simply a log Calabi--Yau variety.
If $X\rightarrow Z$ is the identity and $z\in Z$ a closed point, 
then a log Calabi--Yau morphism around $z\in Z$
is simply a log canonical type singularity structure 
around the point $z$.

A $N$-complement for the projective contraction $X\rightarrow Z$
is a boundary $B$ for which $(X,B)$ is log canonical
and $N(K_X+B)\sim_Z 0$.
}
\end{definition}

\begin{definition}
{\em 
Let $X\rightarrow Z$ be a Fano type morphism. 
Let $(X,B)$ be a log Calabi--Yau structure over $Z$. 
Write $B=\sum_{i=1}^k a_i B_i$, where each $B_i$ is a 
prime reduced divisor and each $a_i$ are nonnegative rational numbers.
We denote by $|B|$ the sum of the coefficients $a_i$.
Then, the {\em relative complexity} of $(X,B)$ over $Z$ is defined to be
\[
c((X,B)/Z):=
\dim X +\rho(X/Z) - |B|.
\] 
The {\em local complexity} of $(X,B)$ over a closed point $z\in Z$ is defined to be
\[
c((X,B);z):= \dim X + \rho(X_z) - |B|,
\]
where $X_z$ is obtained by a base change to the localization of $Z$ at $z$.
We define the {\em absolute relative complexity} to be:
\[
c(X/Z) := \min \left\{ 
c((X,B);z) \mid \text{ $(X,B)$ is log Calabi--Yau over $Z$}\right\}.  
\]
Of course, the dimension and the relative Picard are fixed. 
Hence, the previous definition just intends to maximize the sum of the boundary coefficients 
among log Calabi--Yau structures on $X$ over $Z$.
In the case that $X/Z$ admits no log Calabi--Yau structure, 
then we just set the complexity to be $\dim X+\rho(X/Z)$.
}
\end{definition}

Note that for a toric morphism $X\rightarrow Z$
we can always consider the toric boundary $B$ of $X$
and $(X,B)$ is log Calabi-Yau.
In this case, $B$ is reduced and the number of its
components is exactly $\rho(X)+\dim X$. 
Hence, the absolute relative complexity is zero
around the image of every minimal log canonical center of $(X,B)$. 
In general, it can be proved that toric morphisms
have relative complexity zero everywhere.
In~\cite{Sho00}, Shokurov conjectured that the complexity was nonnegative
and if it was zero, then the morphism was indeed formally toric.
In~\cite{MS21}, the authors settle this conjecture, obtaining the following result:

\begin{theorem}\label{thm:complexity-toric}
Let $X\rightarrow Z$ be a projective contraction.
Then, we have that $c(X;z)\geq 0$.
If $c(X;z)=0$, then formally around 
$z$ the morphism $X\rightarrow Z$ is toric.
Furthermore, if $c((X,B);Z)=0$, 
then $(X,\lfloor B\rfloor)$ 
together with the morphism $X\rightarrow Z$
are formally toric over $z$.
\end{theorem}

It is worth mentioning that
in the local setting
the previous theorem gives
a characterization of toric singularities
using the language
of log canonical singularities.
The previous theorem was obtained based on the work of
many other mathematicians. 
Shokurov proved the theorem for arbitrary morphisms of surfaces in~\cite{Sho00}.
In~\cite{Yao13}, Yao gives a proof of the projective statement for log smooth pairs.
In~\cite{KM99}, Keel and McKernan proved the statement for projective surfaces. 
In~\cite{Pro01b}, Prokhorov proved the statement for certain projective 3-folds using techniques from the minimal model program.
In~\cite{BMSZ18}, the authors prove the projective version of the previous theorem.
The proof of Theorem~\ref{thm:complexity-toric}
used the language of local Cox rings (see, e.g.,~\cite{BM21,BGLM21}).

We expect that the previous theorem 
can be generalized to an inequality
that also considers the coregularity 
of a Fano type pair. 
We propose the following problem.

\begin{problem}\label{problem:coreg-comp}
Let $X$ be a Fano type variety. 
Let $(X,B)$ be a log Calabi-Yau pair.
Then, we have that 
\[
|B|+{\rm coreg}(X) \leq \dim X +\rho(X).
\]
Furthermore, if the equality holds, 
then we have that ${\rm coreg}(X)=0$.
\end{problem}

In other words, the sum of the coefficients
of $B$ is not just bounded by $\dim X+\rho(X)$,
but it is also bounded
by a sharper bound
$\dim X +\rho(X)-{\rm coreg}(X)$.
The previous problem, of course, is more interesting when the coregularity 
is large (or the regularity is small).
For instance, it says that for exceptional Fano varieties
the sum of the coefficients of any complement
is at most $\rho(X)$.

Another possible generalization 
of the complexity is towards the language of generalized pair (see, e.g.,~\cite{Bir21b}).
A generalized pair $(X,B+M)$ is, briefly speaking, a pair with an extra summand $M$
that is the push-forward of a nef divisor
on a higher birational model.
This divisor $M$ is often called the {\em moduli divisor} of the generalized pair,
as it naturally appears in the outcome of the canonical bundle formula (see, e.g.,~\cite{Amb06,Fil20}).
In~\cite{MS21}, the language of generalized pair was used, although the complexity invariant does not depend on the moduli part $M$.
We propose the following problem regarding
generalized pairs
and complexity.

\begin{problem}\label{problem:gen-comp}
Let $X\rightarrow Z$ be a Fano type morphism.
Let $(X,B+M)$ be a generalized log canonical pair so that 
\[
K_X+B+M\sim_{Z,\qq} 0.
\]
Let $|B|$ be the sum of the coefficients of $B$.
Let $M'$ be the model where $M$ descends.
Write $M'=\sum_{i=1}^k a_i M'_i$ where each $a_i$ is a nonnegative rational number
and each $M'_i$ is a nef Cartier divisor
which is not $\qq$-linearly trivial.
We write $|M|$ for the sum of the $a_i$'s.
Then, we have that 
\[
c((X,B+M);Z):=
\dim X + \rho(X) - |B|-|M| \geq 0.
\] 
Furthermore, if the equality holds, then
$X$ is a toric variety.
\end{problem}

Of course, one expects a statement that generalizes both Problem~\ref{problem:coreg-comp} and Problem~\ref{problem:gen-comp}. 
We point out that in the statement of Problem~\ref{problem:gen-comp}, if $B=0$, then 
we expect $X$ to be a special toric variety, similar to a weighted projective space, as in the statement
of Kobayashi-Ochiai Theorem.

The coregularity, by definition,
is a discrete invariant
and can only take integral values.
However, the complexity, on the other hand,
may take fractional values
as shown in the case of $E_8$-singularities. 
In~\cite{BMSZ18}, the authors prove that if the absolute complexity of a projective variety is less than one then it is zero.
In this direction, we know very little about the complexity as an invariant itself, i.e., 
its values larger than one. 

\begin{problem}
Let $n$ be a positive integer.
Describe the set $\mathcal{C}_n$ of absolute complexities
of $n$-dimensional Fano type varieties. 
Does $\mathcal{C}_n$ have accumulation points for some $n$? 
Do the accumulation points comes from absolute complexities in lower dimensions?
\end{problem}

As of today, we do not know examples 
in which the complexity has accumulation points. 
So, maybe the last question is vacuous.
The following example explains the behaviour of the complexity
for klt surface singularities.

\begin{example}
{\em 
Let $(X;x)$ be a klt surface singularity.
Then, we have that 
\[
c(X;x) \in \left\{ 
0,1,\frac{4}{3},\frac{5}{4},\frac{7}{6} 
\right\}. 
\] 
If the absolute complexity is zero,
then the singularity is toric
and the absolute complexity is computed
by the toric boundary which is a $1$-complement.
In the absolute complexity one case, 
the singularity admits a double cover
which is toric. 
The complexity is computed by a reduced
$2$-complement which is unibranch through the singularity. 
In the case that
$c(X;x)\in \{ \frac{4}{3},\frac{5}{4},\frac{7}{6}\}$,
the singularity is of
$E_6,E_7$ or $E_8$ type, respectively.
In these cases, 
the complexity is computed
by a $3,4$ or $6$-complement
which is unibranch through the singularity, respectively.
}
\end{example}

Following the philosophy that 
complements should be lifted from 
minimal log canonical centers,
the following question seems to be natural.

\begin{problem}
Show that the set of absolute complexities
of klt singularities
of coregularity at most $c$ 
is contained in the set
of absolute complexities 
of klt pairs with standard coefficients
of dimension at most $c$.
\end{problem}

The situation of complexity is quite similar to such of coregularity.
Both invariants are minimized in the toric case. 
The second invariant characterizes toricness. 
The first invariant, on the other hand, 
could be zero and the pair not be even 
birationally toric.
Among complements computing the coregularity, 
those which minimize the complexity seem 
to be of special interest. 
We propose the following question.

\begin{question} 
Let $X$ be a Fano type variety.
Can we find a complement 
that computes both
the coregularity and the complexity?
Furthermore, can we find this complement effectively? i.e., to be a $N_c$-complement
where $c$ is the coregularity of $X$.
\end{question} 

The problem of finding 
a complement computing the coregularity
often reduces to finding boundaries
with large coefficients (or multiplicities). 
This gives a very naive connection between the coregularity and the complexity.

We recall the following problem 
introduced in~\cite{BMSZ18}. 
It is inspired
by the classification of DuVal singularities.

\begin{problem}
Let $X$ be a variety of absolute complexity strictly less than two.
Prove that $X$ admits a two-to-one cover
which is rational. 
\end{problem}

In~\cite[Section 7]{BMSZ18},
the authors give an example
of an irrational 
example of absolute coregularity zero. 
This example is a conic bundle
over $\pp^1\times \pp^1$ 
quotient by an involution.

\subsection{Exceptional Fanos and singularities}
In this subsection, we further discuss
exceptional Fano type varieties
and exceptional singularities.

We recall that a Fano type variety
$X$ is said to be {\em exceptional}
if for every complement $(X,B)$
the pair $(X,B)$ has klt singularities.
In other words, it is not possible to produce
a non-trivial log canonical center 
with a complement on $X$. 
Note that we always assume that $X$ is a log
canonical center of $(X,B)$, but this can be regarded
as a trivial log canonical center.
On the other hand, we say that 
a klt singularity $(X;x)$ is {\em exceptional}
if for every complement $(X,B;x)$
there exists at most one log canonical place.
In particular, if $(X;x)$ is an exceptional klt singularity 
and $(X,B;x)$ is a strictly log canonical complement,
then the dlt modification of $(X,B;x)$ is indeed a plt blow-up
for $(X;x)$.
Observe that in the local setting we can always
produce a log canonical center through $x\in X$, 
hence the exceptional case correspond to such 
in which we can produce only one log canonical place. 
Note that the definition of exceptional singularities
does not rule out the possibility that two 
different complements of $(X;x)$ may have different log canonical places.
However, the following lemma implies that this is not the case.

\begin{lemma}\label{lem:unique-plt}
Let $(X,B;x)$ be a klt singularity. 
Assume that it is exceptional.
Then, there is a unique exceptional divisor $E$ over $X$ with center
$c_X(E)=x$ 
so that every strictly log canonical complement of $(X,B;x)$
has $E$ as its unique log canonical center.
\end{lemma}

The previous lemma indicates that an exceptional klt singularity
admits a unique plt blow-up. 
Furthermre, the log pair obtained by adjunction to the exceptional
of such plt blow-up is an exceptional log Fano pair.
The same argument proves the following lemma.

\begin{lemma}\label{lem:from-ex-Fano-to-ex-sing}
Let $(E,B_E)$ be a log Fano pair with standard coefficients.
Assume that $(E,B_E)$ is exceptional.
Let $(X;x)$ be the klt singularity obtained by taking the cone
on $E$ with respect to an ample $\qq$-divisor $A$ satisfying the following conditions:
\begin{itemize}
    \item we have a $\qq$-linear equivalence
    $A\sim -r(K_E+B_E)$, and 
    \item for every prime divisor $P$ on $X$,
    the Weil index of $A$ at $P$ equals the Weil index of $B_E$ at $P$.
\end{itemize}
Then, the klt singularity $(X;x)$ is exceptional. 
\end{lemma}

Thus, the plt blow-up and cone construction 
explained in subsection~\ref{subsec:fano-vs-klt}
preserve the exceptional condition.
In general, finding exceptional Fano varieties
is harder than producing 
exceptional log Fano pairs with standard coefficients.
For instance, in dimension one, 
there is no exceptional Fano variety.
However, we have three exceptional log Fano pairs
with standard coefficients: 
\[
\left(\pp^1,\frac{1}{2}\{0\}+\frac{2}{3}\{1\}+\frac{2}{3}\{\infty\}\right), \quad
\left(\pp^1, \frac{1}{2}\{0\}+\frac{2}{3}\{1\}+\frac{3}{4}\{\infty\} \right), \text{ and }
\left(\pp^1, \frac{1}{2}\{0\}+\frac{2}{3}\{1\}+\frac{4}{5}\{\infty\} \right)
\] 
These three correspond to the $E_6,E_7$ and $E_8$ singularities.
In general, given a Fano variety $X$ it is often the case
that we can find a boundary with standard coefficients $\Delta$
for which $(X,\Delta)$ is an exceptional log Fano pair.
It is fairly simple to produce examples with $X=\pp^n$. 
Then, using Lemma~\ref{lem:from-ex-Fano-to-ex-sing}
we can produce exceptional klt singularities. 
However, even simple questions about exceptional quotient singularities remain open.
In dimension $7$ there is no exceptional quotient singularities~\cite{CS11b}.

\begin{conjecture}\label{conj:quot-excep}
For every dimension $n$ there is an exceptional quotient singularity
of dimension larger than $n$.
\end{conjecture}

We also expect that it is possible to classify exceptional Brieskorn singularities. 

\begin{problem}
Classify exceptional Brieskorn singularities.
\end{problem}

Some progress towards the previous problem is achieved in~\cite{IP01}.
In general, finding examples of exceptional Fano varieties is more challenging. 
In~\cite[Example 21.3]{KM99},
the authors give an example
of a smooth exceptional Fano surface.
We have no standard procedure to create examples
of exceptional Fano varieties in every dimension.

\begin{problem}
Find examples, in every dimension, of exceptional Fano varieties.
Can we produce non-rational examples in every dimension larger than $2$?
\end{problem}

We conclude this subsection by giving some positive evidence for Problem~\ref{problem:coreg-comp}. 

\begin{theorem}
Let $E$ be an exceptional Fano variety
for which ${\rm Cl}(E)\simeq \zz$. 
Let $(E,B_E)$ be a complement.
Then, the sum of the coefficients of $B_E$ is
strictly less than one.
\end{theorem}

\begin{proof}
Assume by contradiction
that the sum of the coefficients of $B_E$ is larger than or equal to one.
Let $A$ be an ample generator of ${\rm Cl}(E)$.
Let $(X;x)$ be the cone over $E$ 
with respect to the $\qq$-polarization 
induced by $A$.
By Lemma~\ref{lem:from-ex-Fano-to-ex-sing}, the singularity
$(X;x)$ is an exceptional singularity.
Let $\pi\colon Y \rightarrow X$ be the blow-up of the maximal ideal at $x$.
Then, $\pi$ is a plt blow-up and
the exceptional divisor is isomorphic to $E$.
Let $B$ be the cone over $B_E$.
Write $B=\sum_{i=1}^k b_iB_i$ 
where each $b_i$ is a positive rational number
and each $B_i$ is a Cartier divisor through $x$. 
We can write 
$B_i={\rm div}(f_i)$ locally around $x$ for each $i\in \{1,\dots,k\}$.
Let 
\[
\Delta:={\rm div}\left( 
\lambda_1f_1+\dots \lambda_kf_k
\right) 
\]
where the $\lambda_i\in \kk$ are general enough.
Let $Z\rightarrow X$ be a log resolution of $(X,B;x)$.
Then, for every divisor $E\subset Z$ 
exceptional over $X$, we have that
\[
{\rm mult}_E(\lambda_1f_1+\dots \lambda_kf_k) \leq {\rm mult}_E(B).
\] 
Hence, we conclude that
for every choice of $0<\lambda_i<b_i$, the pair
\begin{equation}\label{eq: klt-ex-pair}
\left( 
X,\lambda\Delta+\sum_{i=1}^k (b_i-\lambda_i)B_i 
\right) 
\end{equation} 
is a klt pair.
By Proposition~\ref{prop:from-klt-to-Fano}, we know that
the pair~\eqref{eq: klt-ex-pair}
admits a plt blow-up at $x\in X$.
Note that this plt blow-up
is also a plt blow-up of $(X;x)$.
By Lemma~\ref{lem:unique-plt}, 
we know that $(X;x)$ admits a unique plt blow-up.
Hence, the plt blow-up of the pair~\ref{eq: klt-ex-pair}
must be $\pi$. 
In particular, we conclude that 
the pair
\begin{equation}
\left( 
Y, \lambda \Delta_Y +\sum_{i=1}^k(b_i-\lambda_i)B_{i,Y}+E
\right) 
\end{equation}
is plt for every choice of 
$0<\lambda_i <b_i$.
By assumption, we have that 
$\sum_{i=1}^kb_i\geq 1$.
Then, the pair
$(Y,E+(1-\epsilon)\Delta_Y)$
is plt 
and $K_Y+E+(1-\epsilon)\Delta_Y$ is anti-ample over $X$,
for every $\epsilon$ small enough.
In particular, the pair
\[
\left(E,(1-\epsilon)\Delta_Y|_E\right)
\] 
is log Fano for every $\epsilon$ small enough.
By Theorem~\ref{thm:boundedness-complements}, there is a $N$-complement for $(E,(1-\epsilon)\Delta_Y|_E)$.
Hence, we may find a log Calabi-Yau pair $(E,\Delta_E+\Gamma_E)$ where the coefficients of $\Gamma_E$ belong to $\zz\left[\frac{1}{N}\right]_{>0}$
and $\Delta_E:=\Delta_Y|_E$ is reduced.
This contradicts the fact that $E$ is exceptional.
Indeed, every prime divisor in the support of $\Delta_E$ is a log canonical center of $(E,\Delta_E+\Gamma_E)$.
\end{proof}

\subsection{Explicit computations}
As we have discussed above,
del Pezzo surfaces have coregularity zero
if and only they have degree at least two.
There are many more classes
of $3$-dimensional Fano manifolds.
The following question seems natural:

\begin{question}
Classify terminal Fano $3$-folds 
of Picard rank one with coregularity zero.
\end{question}

Given a positive answer 
for Conjecture~\ref{conj:structural-coreg-0},
in order to prove that 
a terminal Fano $3$-fold of Picard
rank one has coregularity zero,
it suffices to study the linear systems
$|-K_X|$ and $|-2K_X|$.
Hence, this problem simplifies
considerably if a basis for
$H^0(-K_X)$ and $H^0(-2K_X)$ are found.

Let us mention that the previous question 
will help to get a complete classification
of $3$-folds of coregularity zero. 
Indeed, if $X$ has coregularity zero,
up to a two-to-one cover, we may find a 
$1$-complement $(X,B)$. 
Let $(Y,B_Y)$ be the dlt modification
of $(X,B)$. 
Then, the variety $Y$ has canonical singularities. 
Replacing $Y$ with a terminalization,
we may assume itself is terminal.
Then, we may run a minimal model program for $K_Y$. 
If it terminates with a Mori fiber space
to a positive dimensional base, then the base 
also has coregularity zero.
One can study the structure of the base
and general fiber to understand the 
structure of the Mori fiber space itself.
In this case, we expect that the Mori fiber space 
to be similar to a product of the base and the general fiber.
On the other case, the MMP
terminates with a terminal Fano variety 
of Picard rank one with a $1$-complement of coregularity zero, i.e., 
one of the varieties as in the previous question.

In dimension two, we know that there are only two
possible birational classes of coregularity zero Fano varieties.
The ones that admit $2$-complements
and the ones that admit $1$-complements.
The latter class is indeed crepant birational
to the projective space with three lines.
It is natural to ask how many birational classes
of coregularity zero Fano varieties are there.

\begin{question}
Classify, up to crepant birational equivalence,
all $3$-folds of coregularity zero. 
\end{question}

The previous question will
help us understand which triangulations
can we find in the sphere $S^2$
among dual complexes of 
log Calabi--Yau $3$-folds of coregularity zero.
Indeed, every log Calabi--Yau $3$-fold
of coregularity zero
is crepant birational to a Fano type variety.
In the case of dimension $4$ 
we do not have a classification
of Fano varieties of
Picard rank one and terminal singularities.
Hence, this approach may only be successful in dimension $3$.
However, a better understanding in dimension three
may allow us to conjecture some behavior in higher dimensions.

We recall that in the toric case, 
smooth Fano varieties are
in correspondence with the so-called
smooth Fano polytopes.
There are exactly
$5$ smooth Fano toric surfaces, 
$18$ smooth Fano toric $3$-folds,
and $123$ smooth Fano toric $4$-folds (see, e.g.,~\cite{Bat91}).
It is not known how the 
number of smooth Fano polytopes
grows with the dimension.

Classically, the invariant that has been used to classify Fano
varieties is the index.
If a Fano variety
of dimension $n$ 
has index larger than $n$, then it is
isomorphic to the projective space.
Fano varieties of index $n-1$
have been classified by Fujita~\cite{Fuj90}
and Fano varieties of index $n-2$
have been classified by Mukai~\cite{Muk89}.
We expect 
Fano varieties of large index
to have smaller coregularity.
Indeed, 
whenever the index is large, we can already find
some interesting elements in $|-K_X|$ (see, e.g.,~\cite{Ale91}).
This raises the following question.

\begin{question}
Is there an inequality involving the 
coregularity and the index of a Fano variety? 
\end{question}

Iskovskikh and Manin proved that a smooth quartic threefold
is birationally superrigid~\cite{IM71}.
Many Fano complete intersections of index one
are proved to be birationally superrigid~\cite{dFEM03}. 
Pukhlikov conjectured that any 
Fano variety of Picard rank one 
and index one is birationally superrigid.
In~\cite{Cas07}, Castravet gives a counter-example for this conjecture.
We recall that the example given by Kaloghiros in~\cite{Kal20}
has coregularity zero and 
it is indeed a birationally rigid example.
It would be interesting to construct more
examples of coregularity zero varieties
which are birationally rigid. 

\begin{problem}
For each dimension $n$, construct a birationally rigid 
Fano variety of dimension $n$ and coregularity zero.
\end{problem}

Among Fano varieties, 
Grassmannians are one of the most 
well-known examples.
It would be interesting to compute
the coregularity explicitly 
for these examples.

\begin{problem} 
For $1\leq k<n$, compute the coregularity
of the Grassmannian $G(k,n)$.
\end{problem} 

An answer to the previous question may also
help to compute the coregularity of $3$-folds.
Indeed, in~\cite{Gus82}, Gushel noted
that some Fano $3$-folds of Picard rank one
are contained in $G(2,m)$.
Using this description, Gushel tried to 
classify Fano $3$-folds. 
Finally, the following question 
is natural but probably is quite challenging.

\begin{question}
Let $n$ be a positive integer
and $d$ be a positive integer less than $n+1$.
Find the coregularity of a general
hypersurface $H_d$ of degree $d$ in $\pp^n$.
\end{question}  

\subsection{Thresholds}

In singularity theory,
thresholds are important invariants of singularities.
In the realm of birational geometry, 
the log canonical threshold
plays a fundamental role in many 
characterizations 
and proofs. 
In a few words, it measures how many times
we can add a divisor through the singularities
so the pair still has log canonical singularities. 
Let $X$ be a normal quasi-projective variety
and $\Delta$ be a boundary for which $(X,\Delta)$ is log canonical.
Let $\Gamma$ be a $\qq$-Cartier divisor on $X$.
\[
{\rm lct}((X,\Delta);\Gamma):=\sup \{ t \mid 
(X,\Delta+t\Gamma) \text{ has log canonical singularities }\}. 
\]
In~\cite{HMX14}, the authors prove that in a fixed dimension,
whenever we control the coefficients of both $\Delta$ and $\Gamma$ in a DCC set, the possible log canonical thresholds is an ACC set.
In what follows, we discuss about local behavior of singularities, 
so we write 
${\rm lct}((X,\Delta;x);\Gamma)$ for the threshold
of the pair around $x$.
We always assume that the divisor $\Gamma$ contains the point $x$.
We define the set of log canonical thresholds
with bounded coregularity to be
\[
{\rm LCT}_{c}(I,J):=\{ t \mid 
t={\rm lct}((X,\Delta;x);\Gamma),
{\rm coeff}(\Delta)\in I,
{\rm coeff}(\Gamma)\in J, 
\text{ and }
{\rm coreg}(X,\Delta+t\Gamma;x)\leq c
\}.  
\]
In~\cite{FMP22}, the authors show that the 
set ${\rm LCT}_c(I,J)$
behaves like the set of log canonical thresholds
with bounded dimension.
In other words, we have the following result.

\begin{theorem}
Let $c$ be a positive integer.
Let $I$ and $J$ be two sets of real numbers
satisfying the descending chain condition.
Then, the set 
${\rm LCT}_c(I,J)$ satisfies the ascending chain condition.
\end{theorem}

The authors also prove
that accumulation points
of log canonical thresholds
with coregularity $c$
come from log canonical thresholds
with coregularity less than $c$.
In the case of coregularity zero,
the only accumulation points
are inherit from the sets $I$ and $J$.
The main idea is to reduce the threshold
computation to a minimal log canonical centers,
which by definition must have dimension at most $c$.
It is expected, that not only canonical threshold,
but other invariants can be reduced to computations
on minimal log canonical centers.

A polytopal version of the log canonical threshold,
the so-called log canonical threshold polytope, 
was introduced by Musta\c{t}\u{a} and Libgober in~\cite{ML11}.
Let $X$ be a normal quasi-projective variety with klt type singularities.
Let $D_1,\dots,D_r$ be a sequence of 
effective $\qq$-Cartier divisors on $X$.
The log canonical threshold polytope is defined to be:
\[
P(X;D_1,\dots,D_r):=
\{ (t_1,\dots,t_r)\in \rr_{\geq 0}^r \mid 
(X,t_1D_1+\dots+t_rD_r) \text{ is log canonical}
\}. 
\]
Analogously as in the previous case, 
we can define log canonical thresholds polytopes
of coregularity $c$. 
This motivates the following question.
\begin{question}
Does the ascending chain condition
for log canonical threshold polytopes
of dimension $r$
and coregularity $c$ 
holds?
\end{question}
In the previous problem, we are fixing 
the number of divisors $r$
so all the polytopes have the same dimension
and they can be compared in the same affine space.
Here, ascending chain condition, means with respect
to the inclusion of polytopes.
In a similar vein, one can ask whether the ascending chain
condition of the volume of such polytopes hold.

An invariant that behaves similar to log canonical thresholds
is the pseudo-effective threshold.
The pseudo-effective threshold
of a variety $X$ with respect to a divisor $D$
is the smallest real number $t$ for which
$K_X+tD$ is pseudo-effective. 
If such a number does not exist, 
then we just say that the pseudo-effective threshold is infinite.
We denote the pseudo-effective threshold by $p(X;D)$.
We can define the set of pseudo-effective thresholds
with coregularity at most $c$ to be 
\[
\mathcal{P}_c(I,J):=
\{ 
p \mid p=p(X;D), 
{\rm coeff}(D)\in I, \text{ and }
{\rm coreg}(X,pD)\leq c 
\}.
\]
In this direction, we propose the following question.

\begin{question}
Let $c$ be a positive integer.
Let $I$ be a set satisfying the descending chain condition.
Does $\mathcal{P}_c(I)$ satisfy the ascending chain condition?
\end{question}

We expect a positive answer to the previous question.
In~\cite{FMP22}, the authors prove a similar statement
for numerically trivial thresholds.
However, we expect the case of pseudo-effective thresholds to 
require some new ideas.

\subsection{K-stability}
The algebraic K-stability theory
uses techniques from algebraic geometry
to decide whether
a Fano variety admits a K\"ahler-Einstein metric.

In~\cite{Fuj19,Li19}, Fujita and Li introduced a 
valuative criterion to detect K-semistability
of K-stability of a Fano variety.
These criteria can be written in terms
of blow-ups of the Fano variety $X$.

\begin{definition}
{\em 
Let $X$ be a $n$-dimensional Fano variety.
Let $p\colon Y\rightarrow X$ be a projective birational morphism
extracting a prime divisor $E$ over $X$.
The {\em expected multiplicity}
$S_E(X)$ of $-K_X$ at $E$ is defined to be
\[
\frac{1}{(-K_X)^n}\int_0^{\infty} {\rm vol}(p^*(-K_X)- tE) dt.
\] 
This invariant is called the expected multiplicity
because somehow it measures the average multiplicity
at the divisor $E$ among 
effective divisos which are $\qq$-linearly 
equivalent to $-K_X$.
The expected multiplicity $S_E(X)$
only depends on $E$ and does not depend on the chosen 
model $Y$.
}
\end{definition}

\begin{definition}
{\em 
Let $X$ be a $n$-dimensional Fano variety.
Let $E$ be a prime divisor over $X$.
The {\em beta invariant} of $E$ is defined to be
the difference between the log discrepancy
at $E$
and the expected multiplicity at $E$, i.e., we define
\[
\beta_E(X) := a_E(X) - S_E(X).
\] 
}
\end{definition}

The following theorem follows from the work of Fujita, Li, Blum, and Xu~\cite{Fuj19,Li19,BX19}.

\begin{theorem}
Let $X$ be a Fano variety.
\begin{itemize}
    \item The variety $X$ is K-semistable
    if and only if $\beta_E(X)\geq 0$ for every $E$ over $X$.
    \item The variety $X$ is K-stable if and only if
    $\beta_E(X)>0$ for every $E$ over $X$.
\end{itemize}
\end{theorem}

It is known that
it suffices to check the condition
$\beta_E(X)\geq 0$ (resp. $\beta_E(X)>0$)
for divisors $E$ which are 
log canonical places
of complements of $X$. 
This means that it suffices to check 
the divisors $E$ for which
$a_E(X,B)=0$ for some $\qq$-complement $B$ of $X$.
Furthermore, in the K-semistable case, there exists a divisor $E_0$ over $X$
which minimizes $\beta$.
The following question is proposed by Xu:

\begin{question}\label{question:xu}
Let $X$ be a Fano variety.
Can we find a complement $(X,B)$
that computes the coregularity of $X$
for which the minimizer of $\beta$
is a log canonical place of $(X,B)$?
\end{question} 

Similar questions can be asked for other invariants
related to K-stability:
the alpha-invariant,
the delta-invariant,
and the normalized volume.
Note that, in the case of Fano varieties of coregularity zero, 
a positive answer to Conjecture~\ref{conj:structural-coreg-0}
and Question~\ref{question:xu}, implies that, 
in order to detect K-semistability
it suffices the linear systems
$H^0(-K_X)$ and $H^0(-2K_X)$
and the associated log canonical places.
Analogously, 
in the case of Fano varieties of coregularity one,
we expect that it suffices to study
$H^0(-NK_X)$ for $N$ at most $6$.
This would considerably simplify the work 
to check the K-stability of Fano varieties of small coregularity.

\subsection{Fundamental groups}
In this subsection, we discuss connections 
of the coregularity with fundamental groups
of klt singularities and Fano type varieties.
We also discuss connections
with log canonical singularities
and log Calabi--Yau pairs. 
We start with the definition of local fundamental groups. 
For simplicity, throughout this subsection, we work
over the field of complex numbers.

\begin{definition}
{\em 
Let $(X,\Delta;x)$ be a singularity of pairs.
The {\em standard approximation} $\Delta_s$ of $\Delta$
is the largest effective divisor $\Delta_s \leq \Delta$
whose coefficients have the form $1-\frac{1}{m}$ for 
positive integers $m$.
If ${\rm coeff}_P(\Delta_s)=1-\frac{1}{m}$, then we say that
$\Delta$ has {\em orbifold index} $m$ at $P$, 
which will be denoted by $m_P$.
For each prime component $P$ of $\Delta$, 
we denote by $\gamma_P$ a loop around $P$.
Given an open subset $U\subset X$ containing $x$, 
we denote by $\pi_1(U,\Delta|_U)$ to be 
the quotient of $\pi_1(U\setminus \supp(\Delta|_U))$
by the smallest normal subgroup generated by the elements
$\gamma_P^{m_P}$.
The {\em regional fundamental group} 
of $(X,\Delta;x)$, denoted by $\pi_1^{\rm reg}(X,\Delta;x)$
is defined to be the inverse limit
of the groups
$\pi_1(U,\Delta|_U)$,
where $U$ runs over all the neighborhoods of $x\in X$.

Analogously, if $(X,\Delta)$ is a quasi-projective pair,
now in the global setting, 
we define $\pi_1^{\rm reg}(X,\Delta)$ to be the inverse
limit of $\pi_1(U,\Delta|_U)$ for all big open subsets
$U\subset X$.
}
\end{definition}

In~\cite{KK14}, Koll\'ar and Kapovich proved that 
any finitely presented group
can appear as the regional fundamental group of a normal algebraic singularity.
However, the situation is much better when we consider singularities 
of the minimal model program. 
In~\cite{Bra20}, Braun proved that the regional fundamental group
of a klt type singularity is finite: 

\begin{theorem}
Let $(X,\Delta;x)$ be a klt type singularity.
Then, the group
$\pi_1^{\rm reg}(X,\Delta;x)$ is finite.
\end{theorem}

This result is based on an analogous projective statement for Fano type varieties. 

\begin{theorem}
Let $(X,\Delta)$ be a projective pair of Fano type.
Then, the group
$\pi_1^{\rm reg}(X,\Delta)$ is finite. 
\end{theorem}

Once the finiteness is settled,
it is natural to try to understand which groups
we can find in a fixed dimension. 
For instance, every finite group $G$ 
appears as the regional fundamental group of a klt type singularity.
Indeed, we can find a representation $G\rightarrow{\rm GL}_n(\cc)$
for which the induced action of $G$ on $\cc^n$ is free in codimension one.
Then, the quotient $X_G:=\cc^n/G$ satisfies that
$\pi_1^{\rm reg}(X_G;0)\simeq G$.
Thus, the regional fundamental groups of quotient singularities
are all possible finite groups. 
On the other hand, if we want to study 
quotient singularities of a fixed dimension, 
we need to understand finite groups in a fixed ${\rm GL}_n(\cc)$.
The following result is due to Camille Jordan 
gives a strong control of such groups.

\begin{theorem}
There exists a constant $c(n)$, only depending on $n$, satisfying the following.
Let $G\leqslant {\rm GL}_n(\cc)$ be a finite subgroup. 
Then, there exists a normal abelian subgroup
$A\leqslant G$ of index at most $c(n)$.
\end{theorem}

Collins proved that for $n\geq 71$, one can indeed take $c(n)=n!$.
In some words, the previous theorem says that finite subgroups
of ${\rm GL}_n(\cc)$ are ``almost abelian" with respect to the dimension.
Note that the previous theorem can also simplify the study
of quotient singularities.
Indeed, the rank of $A$ is at most $n$ and 
the quotient $\cc^n/A$ is a toric singularity.
Thus, any $n$-dimensional quotient singularity admits a cover
of degree at most $c(n)$ which makes it a toric singularity.
As a consequence, we have the following corollary.

\begin{corollary}
There exists a constant $c(n)$, only depending on $n$ satisfying the following.
Let $(X;x)$ be a $n$-dimensional quotient singularity. 
Then, there exists a normal abelian subgroup
$A\leqslant \pi_1^{\rm reg}(X;x)$ 
of index at most $c(n)$ and rank at most $n$.
\end{corollary}

We expect the behavior of klt type singularities
to be similar to such of quotient singularities.
The following theorem gives a realization
of this principle from the perspective of fundamental groups.
This theorem is due to the work of
Braun, Filipazzi, Svaldi, and the author (see, e.g.,~\cite{BFMS20,Mor20c,Mor21}).

\begin{theorem}
Let $n$ be a positive integer. 
There exists a constant $c(n)$, only depending on $n$,
satisfying the following.
Let $(X,\Delta;x)$ be a klt type singularity of dimension $n$
and regularity $r$. 
Then, there exists a short exact sequence
\[
1 \rightarrow A \rightarrow \pi_1^{\rm reg}(X,\Delta;x) \rightarrow N \rightarrow 1,
\]
where $A$ is an abelian group of rank at most $r+1$
and $N$ is a group of order at most $c(n)$.
\end{theorem}

Thus, the ``abelian part" $A$ of the regional fundamental group
of a klt type singularity is determined by its regularity
while the  ``non-abelian part" $N$ is determined by the dimension.
We expect that the abelian part $A$ of $\pi_1^{\rm reg}(X,\Delta;x)$
behaves like the maximal torus of a reductive group.
However, we do not have statements about its maximality.
Even though we have a good structure theorem, 
we do not know how to control the constant $c(n)$.
It is natural to expect that $N$ is some sort of permutation group. 
We propose the following question that may enlighten :

\begin{problem}
For each dimension $n$, construct examples of $n$-dimensional 
klt type singularities $(X,\Delta;x)$ for which
$|\pi_1^{\rm reg}(X,\Delta;x)/A|> n!$ for every
abelian subgroup $A$.
\end{problem}

The previous problem is just asking to check that Collins
bound is not optimal for klt singularities. 
In low dimensions, we even expect that the possible groups 
can be classified: 

\begin{problem}
Classify the regional fundamental groups of klt $3$-fold singularities.
\end{problem}

Rationally connected varieties are quite similar to Fano varieties.
Indeed, every Fano type variety is rationally connected.
On the other hand, every rationally connected variety can 
be birationally transformed into a tower of Mori fiber spaces
whose general fibers are Fano type varieties. 
Thus, it is natural to study the regional fundamental group of
rationally connected varieties. 
In the case of singular rationally connected varieties,
we do not expect the finiteness of the regional fundamental group.
For instance, we have the following example
of a rational surface with rational log canonical singularities and infinite fundamental group:
\begin{example}
{\em 
Let $X=\{y^2-x(x^2-z^2)=0\}\subset \pp^3_{[x:y:z:t]}$, i.e., 
the projective cone over an elliptic curve. 
Observe that $X$ is invariant with respect
to the involution
\[
\tau \colon \pp^3 \rightarrow \pp^3, \text{ with }
\tau([x:y:z:w])=[-x:y:-z:t].
\] 
The fixed points of $\tau$ 
are $[0:0:0:1]$
and four other points.
Let $Y=X/\tau$.
Then $Y$ has four $A_1$ singularities
and a rational singularity $y\in Y$
which is the quotient by an involution
of a simple elliptic singularity.
The quotient $Y$ is rational.
Indeed, it contains the total space of an orbifold bundle
over $\pp^1$ with four orbifold points with isotropy $\frac{1}{2}$.
The regional fundamental group is infinite since it is an extension
of $\zz/2\zz$ and the regional fundamental group
of the cone over an elliptic curve.
}
\end{example}

However, the following conjecture seems to be optimal in this direction:

\begin{conjecture}
Let $X$ be rationally connected.
Let $(X,\Delta)$ be a dlt pair. 
Then, the regional fundamental group
$\pi_1^{\rm reg}(X)$ is finite. 
\end{conjecture}

The following conjecture relates
the fundamental group of rationally connected varieties 
and the coregularity.

\begin{conjecture} 
Let
$X$ be a rationally connected variety of dimension $n$, regularity $r$, and coregularity $c$.
Assume that $G:=\pi_1^{\rm reg}(X)$ is finite.
Then, $G$ admits a subnormal subseries:
\[
A_0 \triangleleft A_1 \triangleleft A_2 \triangleleft G,
\]
such that the following conditions are satisfied: 
\begin{itemize}
    \item $A_0$ acts on either $\mathbb{D}^{r-1}$ or $S^{r-1}$, 
    \item $A_1/A_0$ is an abelian group of rank at most $r$, and 
    \item $A_2/A_1$ has order at most $N(c)$.
\end{itemize}
Here, $N(c)$ only depends on the coregularity $c$.
\end{conjecture}

The previous conjecture can be regarded as follows. 
If $G$ is finite, then the action of $G$ on the universal cover
splits into three different pieces:
the action on the dual complex of a log Calabi--Yau structure, 
the action that fixes point-wise every stratum of the dual complex, 
and the action in the minimal log canonical center of the CY structure.

We recall the following conjecture about the regional fundamental group
of log Calabi--Yau pairs, which is motivated by Campana's conjecture~\cite{Cam11}.

\begin{conjecture}
Let $(X,\Delta)$ be a log Calabi--Yau pair of dimension $n$.
Then, there exists a short exact sequence:
\[
1\rightarrow A\rightarrow \pi_1^{\rm reg}(X,\Delta)\rightarrow N \rightarrow 1, 
\] 
where $A$ is an abelian group generated by at most $2n$ elements
and $N$ is a finite group of order at most $c(n)$. 
\end{conjecture} 

We expect that in some universal cover the action of $A$
is either given by the action of an abelian variety 
or an algebraic torus. 
In other words, the abelian action on a Calabi--Yau variety
should come from a complexification of the circle.
In~\cite{Kol11}, Koll\'ar proved that the fundamental group
of a log canonical singularity is not necessarily finite, 
and we can find the fundamental groups of Riemann surfaces
among them.
It is expected that in dimension at least $5$, 
among regional fundamental groups of lc singularities, we
can find any finitely presented group:

\begin{problem}
Show that any finitely presented group appears
as the regional fundamental group of a log canonical singularity
of dimension $5$.
\end{problem}

On the other hand, we expect that this behavior 
is mostly related to the dual complex. 
Hence, groups acting on log canonical singularities
and fixing the dual complex should be better behaved:

\begin{conjecture}
Let $(X,\Delta;x)$ be a log canonical singularity of dimension $n$.
Then, the normal subgroup
\[
N:= {\rm ker}(\pi_1(X,\Delta;x) \rightarrow \pi_1(\mathcal{D}(X,\Delta;X)))
\] 
is virtually nilpotent of rank at most $2n+1$.
\end{conjecture}

The classification of log canonical surface singularities
is a source of examples that motivates the previous conjecture.
Similar examples can be found by taking orbifold cones
over projective Calabi--Yau surfaces.

\subsection{Degeneration coregularity}

Degenerations of log Calabi--Yau
pairs 
and Fano varieties
are often found in the literature.
In general, being able to 
degenerate a Fano variety
or a Calabi--Yau variety
to a special kind of 
of varieties. 
In many cases, we can use
the structure of the central fiber
to deduce information about nearby fibers.
We define the {\em degeneration coregularity} as follows: 
\[
{\rm degcoreg}(X):=
\min \left\{ 
{\rm coreg}(\mathcal{X}_0) \mid 
\mathcal{X}\rightarrow \mathbb{A}^1
\text{ is a flat family with slc fibers and $\mathcal{X}_t\simeq X$ for some $t$} 
\right\}.
\] 
In the previous definition,
we coregularity of a non-normal pair
is defined to be the minimum coregularity
among the components of its normalization.
One can also define a variation
of the degeneration coregularity
in which we ask the family to be 
isotrivial
outside the central fiber.
A similar definition can be given
for log Calabi--Yau pairs $(X,\Delta)$.
It is clear that 
we have an inequality
\[
{\rm degcoreg}(X,\Delta) 
\leq 
{\rm coreg}(X,\Delta). 
\] 
Furthermore, this inequality may be strict.
For instance, 
if we consider an elliptic curve
$E$, then we have that
\[
0={\rm degcoreg}(E) < {\rm coreg}(E) =1.
\] 
Indeed, we can degenerate our elliptic curve $E$ to a cycle of $k$ rational curves
glued along zero and infinity.

In~\cite{Kul77}, 
Kulikov studied the central
fiber of degenerations of K3 surfaces.
They prove that
for semistable degeneration of K3 surfaces
the central fiber can only have 
thee possible types:
\begin{itemize}
    \item it is a smooth K3 surface, 
    \item the dual complex of the central fiber is a segment of a line, 
    the endpoints correspond to rational surfaces 
    and the interior points correspond to
    ruled surfaces, or 
    \item the dual complex of the central fiber is the triangulation of a sphere, every component of the central fiber is a rational surface of coregularity zero.
\end{itemize}
The three cases can be distinguished in
terms of monodromy around the central fiber.
If the monodromy around the special fiber
is trivial, 
then the central fiber is 
a smooth K3 surface.
In particular, a K3 surface may have
degeneration coregularity zero,
although it has coregularity two. 

In this direction, we propose a question in the opposite direction.
Given a polyhedral complex $\mathcal{P}$
whose gluing functions are linear affine maps and its maximal polyhedra are $n$-dimensional smooth, 
we can associate to it a $n$-dimensional toric 
variety $X(\mathcal{P})$. 
In the previous case, we say that 
$\mathcal{P}$ is a {\em $n$-dimensional linear} dual complex.
In general, the toric variety $X(\mathcal{P})$ may not be simple normal crossing, but
this condition can be obtained by imposing
that every vertex in the dual complex 
is contained in exactly $n+1$
maximal polyhedra. 
In such a case, we say that $\mathcal{P}$ is {\em smooth}.
Finally, if every face of dimension $n-1$ is contained
in exactly two faces of dimension $n$, 
then we say that $\mathcal{P}$ 
is a {\em Calabi--Yau} polyhedral complex. 
The following proposition follows from the previous definitions. 

\begin{proposition}
Let $\mathcal{P}$ be a $n$-dimensional linear snc Calabi--Yau polyhedral complex.
Then, the variety $X(\mathcal{P})$ is a log Calabi--Yau simple normal crossing variety 
so that each component is a projective irreducible toric variety.
\end{proposition}

For instance, if
$\mathcal{P}$ is just a loop with $k$ vertices, 
then the associated toric variety
$X(\mathcal{P})$ is a cycle of $k$ copies
of $\pp^1$ glued along zero and infinity.
Many Calabi--Yau varieties degenerate
to snc pairs of coregularity zero. 
The following question aims to understand the opposite direction:

\begin{problem}
Let $\mathcal{P}$ be a $n$-dimensional linear snc Calabi--Yau polyhedral complex.
Describe the versal deformation space
${\rm Def}X(\mathcal{P})$ in terms of the combinatorics of $\mathcal{P}$.
Can $X(\mathcal{P})$ be deformed into a klt Calabi--Yau variety? 
\end{problem}

We expect the previous question to be more 
accessible when the dual complex of $X(\mathcal{P})$
is a triangulation of a $2$-dimensional sphere.

In a similar vein, we can define the 
{\em iterated degeneration coregularity}. 
We can define this invariant inductively as follows. 
If the degeneration coregularity of $X$ 
equals its dimension $n$, 
then the iterated degeneration
coregularity is just $n$. 
Otherwise, we define it as the minimum
among the 
iterated degeneration coregularity 
of the minimal dlt centers of the central fiber
of any Calabi--Yau semi-log canonical degeneration.
For instance, if we can degenerate
a log Calabi-Yau pair 
so that the minimal log canonical centers
of the central fiber is an elliptic curve, 
then its iterated degeneration coregularity is zero.
Indeed, we can degenerate the elliptic curve
to a semi-log canonical pair
of coregularity zero.
In most cases known to the author, 
the iterated degeneration coregularity
equals the degeneration coregularity. 
It would be interesting to find some examples in which these invariants do not agree.

\begin{question} 
Is there a klt Calabi-Yau variety (or pair)
for which the iterated degeneration coregularity
is not equal to the degeneration coregularity? 
\end{question}

We expect the previous question to have a positive answer. 
However, it would be interesting to have some examples
and understand how the dual complexes
of the different degenerations relate.

In the case that the central fiber of simple normal crossing degeneration is Fano, 
then the structure is much simpler.
Indeed, in this case, the dual complex 
is a simplex of dimension at most $n$, 
where $n$ is the dimension of the general fiber. 
In the case that the dual complex of the central fiber
is a $n$-dimensional simplex, 
the author together with Loginov
proved that each component 
of the central fiber 
is a generalized Bott tower~\cite{ML20}.
Furthermore, the way that these 
generalized Bott towers glue is unique. 
This kind of varieties have coregularity zero.
We propose the following problem 
which is the coregularity one version of the
result due to the author and Loginov.

\begin{problem}
Classify log smooth Fano pairs $(X,\Delta)$
for which $\Delta$ is reduced
and ${\rm coreg}(X,\Delta)$ equal to one.
\end{problem}

If the Picard rank of $X$ 
equals one,
then we expect that the varieties
in the previous problem are close to being toric.
Indeed, we can consider a complement of $(X,\Delta)$ and observe that its complexity is at most two.
However, this may not be the case when the Picard rank of $X$ is higher.

\subsection{Mirror symmetry}
Mirror symmetry is an important topic in the geometry
of Calabi--Yau varieties
and Fano varieties.
It is expected that every Calabi--Yau variety $X$
has a mirror $X^\vee$ whose complex structure
behaves like the symplectic structure of $X$
and vice-versa.
In the case of Fano varieties, the mirror
is an affine variety with a potential,
the so-called Landau-Ginzburg models. 
In many cases, this model can be compactified
into a log Calabi--Yau variety.
In general, we expect the mirror
of a log Calabi--Yau variety
to be another log Calabi--Yau variety.
This has been proved to be the case in 
dimension two~\cite{GHK15}.

The Strominger-Yau-Zaslow conjecture (known as SYZ)
states that a Calabi-Yau manifold $X$ can be fibered
into simpler objects: special Lagrangian tori, 
such that the dual $X^\vee$ can be obtained
by dualizing this family of Lagrangian tori.

Among the many approaches to Mirror symmetry, 
we find the Gross-Siebert program.
The Gross-Siebert program is an algebraic analog
of the SYZ conjecture.
In this program, the Calabi-Yau variety 
is degenerated into a toric variety 
in order to construct its mirror from the central fiber
of the degeneration.
More precisely, 
given a Calabi-Yau variety $X$, 
we try to find a flat degeneration 
$\mathcal{X}\rightarrow \mathbb{A}^1$ for which
$\mathcal{X}_t\simeq X$ for some $t$
and $\mathcal{X}_0$ is toric.
Then, the mirror of $\mathcal{X}_0$
can be constructed using toric geometry.
Finally, we can study appropriate deformations
of the toric mirror $\mathcal{X}_0^\vee$
in order to find the mirror $X^\vee$ of $X$.
In our language, the Gross-Siebert program intends 
to prove Mirror symmetry
for Calabi--Yau's of degeneration coregularity zero.
In the toric setting, the dual of a toric variety
is always toric. 
This leads to our first question regarding 
mirrors of log Calabi--Yau variety.

\begin{question}
Let $(X,\Delta)$ be a log Calabi--Yau variety.
Let $(X^\vee,\Delta^\vee)$ be its Mirror. 
Does the equality 
\[
{\rm coreg}(X,\Delta)={\rm coreg}(X^\vee,\Delta^\vee)
\]
holds?
\end{question}

We expect that the previous inequality should hold. 
This expectation is based on the Mirrors 
of log Calabi-Yau varieties that we can construct
in low dimension.
Specially, among those that are complete intersections
in weighted projective spaces.
Let us mention that Mirror symmetry is also expected
to preserve certain flat deformations and degenerations
of log Calabi--Yau pairs. 
Hence, it is natural to expect that the previous equality
also hold if we replace
the coregularity
with the degeneration coregularity.
Indeed, the Gross-Siebert program can be understood as a
very explicit manifestation of this phenomenon.
In such a case, there are two other natural questions
that arise from a positive answer to the previous question. 

\begin{question}
Let $(X,\Delta)$ be a log Calabi--Yau pair
of dimension $n$ 
and coregularity zero.
Assume Conjecture~\ref{conj:dual-comp}.
Let $(X^\vee,\Delta^\vee)$ be its mirror.
How do the triangulations
of the spheres $S^{n-1}$ given
by the dual complexes relate?
\end{question}

Finally, it is natural to compare
the minimal log canonical centers between
the mirrors. 
However, in order to do so, we need to pass
to a dlt modification.
The minimal dlt centers
of a log Calabi--Yau pair is only well-defined
up to birational equivalence. 
Hence, the best expectation in this direction is the following:

\begin{question}
Let $(X,\Delta)$ be a log Calabi--Yau variety.
Let $(X^\vee,\Delta^\vee)$ be its mirror.
Is a minimal dlt center
of $(X,\Delta)$
birational 
to the mirror of a minimal dlt center of
$(X^\vee,\Delta^\vee)$?
\end{question}

There is little to no evidence 
for a positive answer to the previous question. 

\subsection{Positive characteristic}
Recently, there has been a lot of progress
in the minimal model program
in positive characteristic (see, e.g.,~\cite{HW20}).
Hence, more tools to tackle problems
about complements
on Fano varieties
become accessible.
However, the lack of Kawamata-Viehweg vanishing~\cite{Ber17}
still imposes a big constraint to 
mimic proofs from characteristic zero.
Hence, new ideas are often required when tackling 
these problems in positive characteristics.
In the case of Fano geometry,
we have a couple of examples 
in positive characteristic 
that do not lift to characteristic zero~\cite{Tot17}. 
However, it seems that these examples
have positive coregularity.
We propose the following question:

\begin{question}
Find Fano varieties of coregularity zero
in positive characteristic 
that do not lift to 
characteristic zero.
\end{question}

Many results for complements,
coregularity,
and dual complexes,
are expected to hold in dimension 
three when the characteristic is larger or
equal to $5$.
However, there could be some interesting
examples in low characteristic that do not lift
to characteristic zero. 

\begin{question}
In dimension $3$, can we find a 
log Calabi--Yau pair 
$(X,\Delta)$ with $X$ Fano, 
such that the triangulation
of the dual complex
$\mathcal{D}(X,\Delta)$
can not be obtained in characteristic zero?
\end{question}

By the Lefschetz principle,
every PL-manifold obtained 
as $\mathcal{D}(X,\Delta)$ for some Calabi--Yau pair
in characteristic zero,
is already obtained over the complex numbers. 
The previous question in some sense
aims to understand if the theory
of complements can give a
combinatorial obstruction for the lifting
to characteristic zero.

\subsection{Fano rings}
The definition
of Fano varieties
is a projective one. 
On the other hand,
the definition
of klt singularities is a local one,
either in the \'etale topology
or analytic topology. 
One can define affine Fano varieties as follows:

\begin{definition}
{\em 
An affine variety $U$ is said to be
an {\em affine Fano variety} 
if there is a 
projective compactification $U\hookrightarrow X$
such that if we define 
$\Delta:=X\setminus U$ with its reduced structure,
the pair
$(X,\Delta)$ is log Fano
with log canonical singularities.
}
\end{definition}

We can give an analogous definition
for {\em affine Calabi--Yau variety}. 
These varieties are also called 
open Fanos and open Calabi--Yau's respectively in the literature.

\begin{definition}
{\em 
We say that a ring $R$ is a {\em Fano ring}
(resp. {\em Calabi--Yau ring})
if ${\rm Spec}(R)$ is an affine Fano variety
(resp. affine Calabi--Yau variety). 
}
\end{definition}

In dimension one, the only
Fano ring is $\kk[x]$.
On the other hand, 
the only Calabi--Yau ring is $\kk[x^{\pm 1}]$.
In higher dimensions, due to blow-ups, 
the classification of Fano rings 
and Calabi--Yau rings is probably a difficult task.
However, we expect that the situation
is somehow managable in dimension $2$, so we propose the following problem:

\begin{question}
Classify Fano rings and Calabi--Yau rings of dimension $2$.
\end{question}

It would be interesting to have a purely commutative algebra 
definition of Fano rings and Calabi--Yau rings.
As the previous definitions heavily rely on projective geometry.
In the study of coregularity zero Fano varieties, 
affine Fano varieties of coregularity zero 
will play the role that complex tori plays in toric geometry.
Thus, a better understanding of these rings is desirable.

\subsection*{Acknowledgements} 
The author would like to thank Priyankur Chaudhuri, 
Fernando Figueroa, Christopher Hacon, Mirko Mauri, Junyao Peng, and Talon Stark, 
for many discussions and comments on the first draft of this manuscript.

\bibliographystyle{habbvr}
\bibliography{mybibfile}

\vspace{0.5cm}
\end{document}